\documentclass[a4paper]{amsart}
\date{November 23, 2019}
\usepackage[dvipdfmx]{graphicx}
\usepackage{latexsym}
\usepackage{amssymb}
\usepackage{amsmath}
\usepackage{amscd}
\usepackage{amsthm}
\usepackage{amsfonts}
\usepackage{mathrsfs}
\usepackage{enumerate}
\usepackage{enumitem}
\usepackage{bm}
\usepackage[usenames]{color}
\usepackage[active]{srcltx}
\usepackage[noadjust]{cite}

\title{%
Mixed type surfaces with
bounded Gaussian curvature 
in three-dimensional Lorentzian manifolds
}
\author{A.~Honda}
\address{%
(A.~Honda)
   Department of Applied Mathematics, 
   Faculty of Engineering, Yokohama National University, 
   79-5 Tokiwadai, Hodogaya, Yokohama 240-8501, Japan
}
\email{honda-atsufumi-kp@ynu.ac.jp}
\author{K.~Saji}
\address{%
(K.~Saji)
   Department of Mathematics, 
   Kobe University,
   Rokko 1-1, Nada, 
   Kobe 657-8501, Japan
}
\email{saji@math.kobe-u.ac.jp}
\author{K.~Teramoto}
\address{%
(K.~Teramoto)
   Institute of Mathematics for Industry, 
   Kyushu University,
   Motooka 744, 
   Fukuoka 819-0395, Japan
}
\email{k-teramoto@imi.kyushu-u.ac.jp}
\thanks{This work was supported by 
JSPS KAKENHI Grant Numbers 16K17605, 18K03301, 17J02151.}
\usepackage{amsthm}
\theoremstyle{plain}
 \newtheorem{introtheorem}{Theorem}
 \newtheorem{introcorollary}[introtheorem]{Corollary}
 
 \newtheorem{theorem}{Theorem}[section]
 \newtheorem{proposition}[theorem]{Proposition}
 
 \newtheorem{fact}[theorem]{Fact}
 \newtheorem{lemma}[theorem]{Lemma}
 \newtheorem{corollary}[theorem]{Corollary}

\theoremstyle{definition}
 \newtheorem{definition}[theorem]{Definition}
\theoremstyle{remark}
 \newtheorem{remark}[theorem]{Remark}
 \newtheorem*{remark*}{Remark}
 \newtheorem{example}[theorem]{Example}
 \newtheorem*{acknowledgement}{Acknowledgement}

\numberwithin{equation}{section}

 \makeatletter
    
    \@addtoreset{equation}{section}
  \makeatother
\newcommand{\vect}[1]{\boldsymbol{#1}}

\newcommand{\inner}[2]{\left\langle{#1},{#2}\right\rangle}

\newcommand{\ep}{\varepsilon}
\newcommand{\R}{\boldsymbol{R}}
\newcommand{\N}{\boldsymbol{N}}
\newcommand{\Z}{\boldsymbol{Z}}
\newcommand{\C}{\boldsymbol{C}}

\newcommand{\sgn}{\operatorname{sgn}}

\begin{document}
\begin{abstract}
A {\it mixed type surface\/}
is a connected regular surface in a Lorentzian 3-manifold
with non-empty spacelike and timelike point sets.
The induced metric of a mixed type surface
is a signature-changing metric,
and their lightlike points may be regarded 
as singular points of such metrics.
In this paper, we investigate the behavior of Gaussian curvature 
at a non-degenerate lightlike point of a mixed type surface.
To characterize the boundedness of Gaussian curvature
at a non-degenerate lightlike points,
we introduce several fundamental invariants along
non-degenerate lightlike points, 
such as the {\it lightlike singular curvature}
and the {\it lightlike normal curvature}.
Moreover, using the results by Pelletier and Steller,
we obtain the Gauss--Bonnet type formula 
for mixed type surfaces with bounded Gaussian curvature.
\end{abstract}
\maketitle

\section{Introduction}

Let $M^3$ be a Lorentzian $3$-manifold. 
A {\it mixed type surface} in $M^3$
is a connected regular surface 
whose spacelike and timelike point sets are both non-empty.
Although mixed type surfaces have no singular points 
as smooth maps, 
we may regard the {\it lightlike points} of a mixed type surface 
as {\it singular points} of the induced metric 
(i.e., the first fundamental form).
Here, a singular point of a metric 
is defined as a point
at which the metric is degenerate.
In this paper, we study 
the lightlike points of mixed type surfaces
in the way similar to the case of 
singular points of smooth maps, especially, of {\it wave fronts}.

\begin{figure}[htb]
\begin{center}
 \begin{tabular}{{c}}
  \resizebox{4.5cm}{!}{\includegraphics{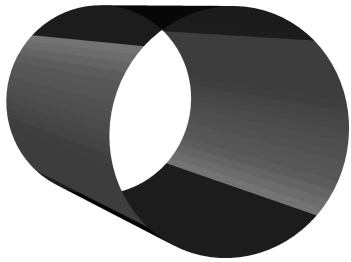}}
 \end{tabular}
 \label{fig:fronts}
 \caption{
 A mixed type surface in 
 the Lorentz-Minkowski $3$-space $\R^3_1$
 given by a cylinder over a circle on the timelike plane
 (cf.\ Example \ref{ex:flat-torus}).
 The dark (resp.\ bright) part is the image 
 of the spacelike (resp.\ timelike) point sets 
 of the surface. 
 }
\end{center}
\end{figure}

\subsection{Motivation}

Consider a mixed type surface
in a Lorentzian $3$-manifold $M^3$
given by an immersion 
$f : \Sigma \to M^3$
of a $2$-manifold $\Sigma$.
The induced metric $ds^2$ of $f$
is a signature-changing metric on $\Sigma$.
The geometry of $2$-manifold endowed 
with a signature-changing metric is studied 
\cite{AFL, KosKri, KosKri1, KosKri2,
IzumiyaTari2010, IzumiyaTari2013, 
GR, Rem-Pseudo, Rem15, Rem-Tari, 
GKT, KT, Mier, Steller, Pelletier}.
Signature-changing metrics are also considered
in the quantum theory of gravitation and general relativity
(cf.\ \cite{Sakhar, AB}).
Mixed type surfaces with zero mean curvature 
(mixed type ZMC surfaces, for short) are studied
\cite{FKKRSTUYY, FRUYY, FKKRSUYY_okayama, 
FKKRSUYY_osaka, FKKRUY_qj, FKKRUY_osaka, FHKUY, 
Gu, Klyachin}.
Mixed type ZMC surfaces are of importance
from the fluid mechanical point of view.
The type-change of causal characters of
mixed type ZMC surfaces
corresponds to the change of a stream function
being from subsonic to supersonic
(for more precise, see \cite{FKKRSUYY_okayama}).
We remark that, in \cite{HKKUY},
it was proved that there do not exist
mixed type surfaces with 
non-zero constant mean curvature 
(see also \cite{UY_geloma, UY_2018}).

On the other hand,
the geometry of singularities of smooth maps, 
in particular of wave fronts,
is under intense investigation.
In \cite{SUY1},
the fundamental invariants of cuspidal edge, 
called the {\it singular curvature} $\kappa_s$ and 
the {\it limiting normal curvature} $\kappa_{\nu}$,
were introduced.
They have
the following significant properties:
\begin{itemize}
\item
The singular curvature $\kappa_s$
affects the shape of cuspidal edge 
\cite[Theorem 1.17]{SUY1}, and 
tends to $-\infty$ when 
the cuspidal edge accumulates to
a swallowtail \cite[Corollary 1.14]{SUY1}.
Moreover, $\kappa_s$ appears in the remainder term 
of the Gauss-Bonnet type formula,
see \cite{SUY1, SUY2, SUY4} (cf.\ \cite{Kos2}).
As proved in \cite[Proposition 1.8]{SUY1}, 
$\kappa_s$ is an {\it intrinsic} invariant
(namely, $\kappa_s$ can be expressed 
in terms of the first fundamental form).
\item
The limiting normal curvature $\kappa_{\nu}$
is closely related to the boundedness of the Gaussian curvature
on wave fronts
(\cite[Theorem 3.1]{SUY1}, \cite[Theorem 3.9]{MSUY}),
and $\kappa_{\nu}$ is an {\it extrinsic} invariant
(\cite[Corollary B]{NUY}).
Unlike the case of the singular curvature,
$\kappa_{\nu}$ can be extended continuously across 
the swallowtail singularity
(\cite[Proposition 2.9]{MSUY}).
\end{itemize}
Based on the criteria for the cuspidal edge and the swallowtail
given in \cite{KRSUY},
the notion of 
the {\it singular points of the first kind} 
(resp.\ {\it of the second kind\/}) 
was introduced in \cite{MSUY},
which can be regarded as an 
generalization of cuspidal edges (resp.\ swallowtails),
reflecting their intrinsic natures
(cf.\ $A_2$- and $A_3$-points in \cite{SUY2, SUY4, SUY5, HHNSUY}).

The purpose of this paper
is to develop the theory of mixed type surfaces
using the method of wave fronts.
In particular,
we introduce several invariants of
lightlike points of the first kind,
such as the {\it lightlike singular curvature} $\kappa_L$
and the {\it lightlike normal curvature} $\kappa_N$, 
and investigate the behavior of 
the geometric quantities, such as Gaussian curvature, of
mixed type surfaces 
near the lightlike points via their invariants.

\subsection{Statement of results}

For more precise, let $f:\Sigma\to M^3$ be a mixed type surface
in a Lorentzian $3$-manifold $M^3=(M^3,\bar{g})$.
A point is said to be a {\it lightlike point}
if the first fundamental form
$ds^2:=f^*\bar{g}$ (cf.\ \eqref{eq:1FF})
degenerates at $p$.
Denote by $LD\,(\subset\Sigma)$ the lightlike point set of $f$,
which is also called the {\it locus of degeneracy}.
On a coordinate neighborhood $(U;u,v)$,
set a smooth function 
$\lambda:=EG-F^2$,
where $ds^2=E\,du^2+2F\,du\,dv+G\,dv^2$.
Then, a point $p\in U$ is lightlike 
if and only if $\lambda(p)=0$.
A lightlike point $p$ satisfying
$d\lambda(p)\ne0$ is said to be {\it non-degenerate}.
By the implicit function theorem, 
the lightlike point set $LD$
can be parametrized by a regular curve
near $p$ (called the {\it characteristic curve\/}). 

A non-degenerate lightlike point $p\in LD$
is said to be {\it of the first kind}
(resp.\ {\it of the second kind\/}) if
$df_p(\vect{v}) \in T_{f(p)}M^3$ is spacelike (resp.\ lightlike)
for each non-zero tangent vector $\vect{v}\in T_pLD$
(Definition \ref{def:types-LPT}).
If $p$ is a lightlike point of the first kind,
then the image $f(LD)$ is 
a spacelike regular curve in $M^3$ near $p$.
A lightlike point $p$ of the second kind
is called {\it an $L_{\infty}$-point}
if the image $f(LD)$ is 
a lightlike regular curve in $M^3$ near $p$.
If not, a lightlike point $p$ of the second kind 
is said to be {\it admissible} 
(cf.\ Definition \ref{def:types-LPT}).
We remark that the notion of 
lightlike points of the first kind
can be considered as an analogue 
for the cuspidal edge singularity of wave fronts,
see Remark \ref{rem:L-conelike}.
We also introduce the notion of 
{\it $L_k$-points} $(k\geq3)$
in Definition \ref{def:types-LPT}.
Such $L_k$-points are lightlike points $p$ of the second kind,
and the notion of $L_3$-points 
can be considered as an analogue 
for the swallowtail singularity of wave fronts.

Then, we define several invariants of
lightlike points of the first kind,
such as the {\it lightlike singular curvature} $\kappa_L$
and the {\it lightlike normal curvature} $\kappa_N$
(see Definition \ref{def:invariants},
cf.\ Proposition \ref{prop:curvature-G}).
Like the case of the singular curvature $\kappa_s$,
the lightlike singular curvature $\kappa_L$
affects the shape of mixed type surfaces (Corollary \ref{cor:shape}).
Moreover, the following holds:

\begin{introtheorem}\label{thm:introA}
Let $f:\Sigma\to M^3$ be a mixed type surface
in a Lorentzian $3$-manifold $M^3$
and $p\in \Sigma$ a lightlike point 
of the admissible second kind.
Then,
\begin{itemize}
\item[{\rm (i)}] 
the lightlike singular curvature $\kappa_L$ tends to $-\infty$ at $p$.
\item[{\rm (ii)}] 
the lightlike normal curvature
$\kappa_N$ converges to $0$ or diverges to $\pm\infty$ at $p$.
In particular, if $p$ is not an $L_3$-point, 
then $\kappa_N$ tends to $-\infty$ at $p$.
\end{itemize}
\end{introtheorem}

By the assertion (i) of Theorem \ref{thm:introA}, 
we have that $\kappa_L$ behaves similarly to 
the singular curvature $\kappa_s$
for the cuspidal edge singularity of wave fronts
(\cite[Corollary 1.14]{SUY1}).
However, 
the assertion (ii) of Theorem \ref{thm:introA}
implies that the behavior of $\kappa_N$ 
is different from that of 
the limiting normal curvature $\kappa_{\nu}$
for the cuspidal edge singularity of wave fronts,
since $\kappa_{\nu}$ can be extended continuously across 
the swallowtail singularity
(\cite[Proposition 2.9]{MSUY}).
We also introduce the invariants 
called the {\it lightlike geodesic torsion} $\kappa_G$
and the {\it balancing curvature} $\kappa_B$
(Definitions \ref{def:invariants} and \ref{def:kappa_B},
cf.\ Proposition \ref{prop:curvature-G}, 
Proposition \ref{prop:kappa_B_adapted}),
and prove the similar results in 
Theorems \ref{thm:kG} and \ref{thm:kB}.

Moreover,
unlike the case of wave fronts,
both $\kappa_L$ and 
$\kappa_N$ are related to the behavior of 
the Gaussian curvature $K$
at a lightlike point.
More precisely, we prove the following:

\begin{introtheorem}\label{thm:introB}
Let $f:\Sigma\to M^3$ be a mixed type surface
in a Lorentzian $3$-manifold $M^3$.
If the Gaussian curvature $K$ of $f$
is bounded on a neighborhood of 
a non-degenerate lightlike point $p\in \Sigma$, 
then $p$ must be of the first kind.
Moreover, for a lightlike point $p$ of the first kind,
$K$ is bounded on a neighborhood $U$ of $p$ 
if and only if 
$$
\kappa_L=0
\quad \text{and} \quad
\kappa_N=\kappa_B
$$
hold along the characteristic curve in $U$.
\end{introtheorem}

We remark that, 
in \cite{HKKUY} and \cite{UY_2018},
it was proved that,
if the mean curvature $H$ of a mixed type surface
is bounded 
at a non-degenerate lightlike point $p$,
then $p$ must be an $L_{\infty}$-point
(cf.\ Fact \ref{fact:HKKUY}, \cite{UY_geloma}).
Hence, the first assertion of Theorem \ref{thm:introB} 
can be regarded as the corresponding result
in the case of Gaussian curvature.

On the other hand, 
in \cite[Corollary B]{NUY},
it was proved that the limiting normal curvature $\kappa_{\nu}$ of 
a cuspidal edge is an extrinsic invariant.
Similarly, 
the lightlike normal curvature $\kappa_N$ of 
a mixed type surface is an extrinsic invariant
\cite{Honda-deform}.
However, 
in the case of vanishing lightlike singular curvature 
$\kappa_L=0$,
we can prove the following.

\begin{introcorollary}\label{cor:introC}
Let $f:\Sigma\to M^3$ be a mixed type surface
in a Lorentzian $3$-manifold $M^3$,
and let 
$p\in \Sigma$ be a lightlike point of the first kind.
If $\kappa_L=0$ holds along the characteristic curve near $p$,
then the lightlike normal curvature $\kappa_N$
is an intrinsic invariant.
\end{introcorollary}

Such a phenomenon does not occur 
in the case of the cuspidal edge singularity
on wave fronts (see Remark \ref{rem:flat-CE}).

In the case of wave fronts in $\R^3$,
a Monge form of the cuspidal edge 
was given in \cite{MS},
which clarifies relationships among the invariants.
Moreover, by using the Monge form, we can provide 
many examples of the cuspidal edge easily.
In Proposition \ref{prop:monge},
we derive a Monge form of a mixed type surface
at a lightlike point of the first kind. 

Finally, in Section \ref{sec:GB},
applying the results by Pelletier \cite{Pelletier} and Steller \cite{Steller}
and Theorem \ref{thm:introB},
we obtain the Gauss--Bonnet type formula 
for mixed type surface with bounded Gaussian curvature.

\begin{introcorollary}\label{cor:introD}
Let $f : \Sigma \to (M^3,\bar{g})$ be a mixed type surface 
in a Lorentzian $3$-manifold $(M^3,\bar{g})$,
where $\Sigma$ is 
a connected compact oriented smooth $2$-manifold
without boundary.
If every lightlike point of $f$ is non-degenerate
and $f$ has bounded Gaussian curvature,
then
$$
  \int_{\Sigma} K\,dA = 2\pi \,\chi (\Sigma)
$$
holds, where $\chi (\Sigma)$ is the Euler characteristic of $\Sigma$.
\end{introcorollary}

\subsection{Organization of this paper}

This paper is organized as follows.
In Section \ref{sec:prelim},
we give fundamental properties
of the lightlike points of mixed type surfaces.
In particular, we show several conditions for lightlike points
to be of the first kind, see Definition \ref{def:types-LPT},
\eqref{eq:image-is-spacelike}, and 
Proposition \ref{prop:type1},
cf.\ Remark \ref{rem:L-conelike}.
In Section \ref{sec:invariant},
we define the invariants, such as 
the lightlike singular curvature $\kappa_L$,
the lightlike normal curvature $\kappa_N$,
the lightlike geodesic torsion $\kappa_G$,
and 
the balancing curvature $\kappa_B$
(see Definitions \ref{def:invariants} and \ref{def:kappa_B}).
In particular, 
we show that $\kappa_L$ is an intrinsic invariant
(Proposition \ref{prop:kappa_L-adap}),
and we prove the formula for $\kappa_B$ 
in terms of the adapted coordinate system
(Proposition \ref{prop:kappa_B_adapted}).
In Section \ref{sec:2nd-kind},
we prove Theorem \ref{thm:introA}.
More precisely, we first consider $\kappa_L$
and prove the assertion (i) 
of Theorem \ref{thm:introA}.
Next, to calculate $\kappa_N$ and $\kappa_G$,
we prepare a formula of the cross product 
at a lightlike point
(Lemma \ref{lem:gaiseki}).
Then, we give a proof of the assertion (ii) 
of Theorem \ref{thm:introA},
and Theorem \ref{thm:kG}.
Finally, we show a similar result about 
$\kappa_B$ in Theorem \ref{thm:kB}.
In Section \ref{sec:Gauss},
we investigate the behavior of 
Gaussian curvature $K$,
and prove Theorem \ref{thm:introB} and 
Corollary \ref{cor:introC}.
We also show that $\kappa_L$
affects the shape of mixed type surfaces (Corollary \ref{cor:shape}),
and 
that the principal curvatures of a mixed type surface 
with bounded Gaussian curvature
are real valued near the lightlike point set
(Corollary \ref{cor:umbilic}).
Then, in Section \ref{sec:monge},
we derive a Monge form of a mixed type surface
at a lightlike point of the first kind 
in the Lorentz-Minkowski $3$-space
(Proposition \ref{prop:monge}). 
Finally, in Section \ref{sec:GB},
reviewing the results by Pelletier \cite{Pelletier} and Steller \cite{Steller},
we prove Corollary \ref{cor:introD}.

\section{Mixed type surfaces}
\label{sec:prelim}

Let $(M^3,\bar{g})$ be an oriented Lorentzian $3$-manifold
with the metric $\bar{g}=\inner{~}{~}$. 
A tangent vector 
$\vect{v}\in T_pM^3$ $(p \in M^3)$
is called {\it spacelike\/}
if $\inner{\vect{v}}{\vect{v}}>0$ or $\vect{v} = \vect{0}$.
Similarly, 
if $\inner{\vect{v}}{\vect{v}}<0$
(resp.\ $\inner{\vect{v}}{\vect{v}}=0$),
$\vect{v}$ is called {\it timelike\/} 
(resp.\ {\it lightlike\/}).
Let $\{\vect{e}_1,\vect{e}_2,\vect{e}_3\}$ 
be a positively oriented orthonormal basis of $T_pM^3$,
namely, 
\begin{equation}\label{eq:ONB}
  \inner{\vect{e}_1}{\vect{e}_1}
  =\inner{\vect{e}_2}{\vect{e}_2}
  =-\inner{\vect{e}_3}{\vect{e}_3}
  =1,\quad
  \inner{\vect{e}_i}{\vect{e}_j}=0
\end{equation}
holds, where $i,j=1,2,3$ $(i\ne j)$.
For tangent vectors $\vect{v}, \vect{w} \in T_pM^3$,
the vector product $\vect{v}\times \vect{w}$
is given by
\begin{equation}\label{eq:CROSS}
  \vect{v}\times \vect{w}
  := \left| \begin{array}{cc} v_2 & w_2\\v_3 & w_3 \end{array} \right| \vect{e}_1
    + \left| \begin{array}{cc} v_3 & w_3\\v_1 & w_1 \end{array} \right| \vect{e}_2
    - \left| \begin{array}{cc} v_1 & w_1\\v_2 & w_2 \end{array} \right| \vect{e}_3,
\end{equation}
where 
$v_i:=\inner{\vect{v}}{\vect{e}_i}$, 
$w_i:=\inner{\vect{w}}{\vect{e}_i}$ 
$(i=1,2,3)$ 
are components of $\vect{v}$, $\vect{w}$
with respect to the orthonormal basis 
$\{\vect{e}_1,\vect{e}_2,\vect{e}_3\}$.
Then, it holds that
\begin{enumerate}
\item
$\vect{v}\times \vect{w}$ is orthogonal to both $\vect{v}$ and $\vect{w}$,
\item 
$  \inner{ \vect{v}\times \vect{w} }{ \vect{v}\times \vect{w} }
  = - \inner{\vect{v}}{\vect{v}} \inner{\vect{w}}{\vect{w}}
     + \inner{\vect{v}}{\vect{w}}^2.$
\end{enumerate}
For $\vect{v}\in T_pM^3$, we set 
$|\vect{v}|:= \sqrt{|\inner{\vect{v}}{\vect{v}}|}$.

In this paper, a {\it surface} in 
a Lorentzian $3$-manifold 
$M^3=(M^3,\bar{g})$
is defined to be an immersion 
$$
  f:\Sigma\longrightarrow (M^3,\bar{g})
$$
of a differentiable connected $2$-manifold $\Sigma$ into $M^3$.
The {\it first fundamental form} 
(or the {\it induced metric}) of $f$
is the smooth metric $ds^2$ on $\Sigma$
defined by
\begin{equation}\label{eq:1FF}
  (ds^2)_p(\vect{v}, \vect{w}) 
  := \inner{df_p(\vect{v})}{df_p(\vect{w})}
  \qquad
  (\vect{v}, \vect{w} \in T_p\Sigma, ~ p\in \Sigma).
\end{equation}
A point $p\in \Sigma$ is called 
a {\it spacelike\/}
(resp.\ {\it timelike\/}, {\it lightlike\/}) {\it point\/},
if $(ds^2)_p$ is a positive definite
(resp.\ indefinite, degenerate)
symmetric bilinear form on $T_p\Sigma$.
We denote by $\Sigma_+$ 
(resp.\ $\Sigma_-$, $LD$)
the set of spacelike 
(resp.\ timelike, lightlike)
points.
A surface $f:\Sigma\to M^3$ is called 
{\it spacelike} (resp.\ {\it timelike}),
if $\Sigma$ coincides with $\Sigma_+$ (resp.\ $\Sigma_-$).
If both the spacelike sets $\Sigma_+$ and 
the timelike sets $\Sigma_-$ are non-empty,
the surface is called a {\it mixed type surface\/}.
 
On a local coordinate neighborhood $(U;u,v)$,
set $f_u:=df(\partial_u)$, $f_v:=df(\partial_u)$,
where 
$\partial_u:=\partial/\partial u$, 
$\partial_v:=\partial/\partial v$.
Then,
$ds^2$ is written as 
\begin{equation}\label{eq:1st-FF}
  ds^2 = E\,du^2 +2F\,du\,dv + G\,dv^2,
\end{equation}
where 
$E:=\inner{f_u}{f_u}$,
$F:=\inner{f_u}{f_v}$ and
$G:=\inner{f_v}{f_v}$.
Setting the function $\lambda$ as
\begin{equation}\label{eq:discriminant}
\lambda := EG-F^2,
\end{equation}
a point $q\in U$ is a lightlike (resp.\ spacelike, timelike) point
if and only if 
$\lambda(q)=0$ (resp.\ $\lambda(q)>0$, $\lambda(q)<0$)
holds.
Namely, 
$U_{\pm}:=\Sigma_\pm\cap U$ are written as
$$
  U_+=\{ q\in U \,;\, \lambda(q)>0 \},\qquad
  U_-=\{ q\in U \,;\, \lambda(q)<0 \}.
$$
We call $\lambda$ the {\it discriminant function}.

\subsection{Non-degenerate lightlike points}

Let $f:\Sigma\to M^3$ be a mixed type surface.
For each point $p\in \Sigma$,
the subspace
$$
  \mathcal{N}_p:=\left\{
    \vect{v}\in T_p\Sigma\,;\,
    (ds^2)_p(\vect{v},\vect{x})=0  
    \text{~holds for any~} \vect{x}\in T_p\Sigma
  \right\}
$$
of $T_p\Sigma$ is called the {\it null space} at $p$.
Then, $\mathcal{N}_p\ne\{ \vect{0} \}$
if and only if $p$ is a lightlike point.
A nonzero element of $\mathcal{N}_p$ is called 
a {\it null vector} at $p$.
We remark that ${\rm dim}\,\mathcal{N}_p=2$ does not 
occur since the image $Q:=df_p(\mathcal{N}_p)$
is a $2$-dimensional degenerate subspace of 
the Lorentzian inner product space $T_{f(p)}M^3$.
Thus, we have
\begin{equation}\label{lem:corank1}
  \dim\mathcal{N}_p=1
\end{equation}
for each lightlike point $p\in LD$ of 
a mixed type surface $f:\Sigma\to M^3$.
Then, a smooth vector field $\eta$ 
defined on a neighborhood $U$ of $p\in LD$
is called a {\it null vector field} if
$\eta_q \in T_q\Sigma$ is a null vector at each 
$q\in LD\cap U$.

As in the introduction,
a lightlike point $p\in LD$ is called {\it non-degenerate}
if $d\lambda(p) \neq0$, where $\lambda$ is a discriminant function.
By the implicit function theorem,
there exists a regular curve
$\gamma(t)$ $(|t|<\ep)$
in $\Sigma$
such that $p=\gamma(0)$
and ${\rm Image}(\gamma)=LD$
holds in a neighborhood of $p$.
We call $\gamma(t)$ 
a {\it characteristic curve}.
For a null vector field $\eta$
defined on a neighborhood of $p$,
the restriction 
$\eta(t):=\eta_{\gamma(t)}\in T_{\gamma(t)}\Sigma$
is called a {\it null vector field along $\gamma(t)$}.

A lightlike point $p\in LD$
is said to be
a {\it type-changing point\/}
if each open neighborhood $U$ of $p$
satisfies $U\cap \Sigma_+\neq \emptyset$
and $U\cap \Sigma_-\neq \emptyset$.

\begin{lemma}
A non-degenerate lightlike point is a type-changing point.
\end{lemma}

\begin{proof}
Since $\gamma$ is regular,
we can take a local coordinate system $(U;u,v)$ 
around $p$
such that $LD=\{(u,0)\}$.
Then, $\lambda(u,0)=0$ yields $\lambda_u(u,0)=0$.
Since $p=(0,0)$ is non-degenerate,
we have $\lambda_v(u,0)\neq0$.
Hence, changing the orientation of $v$-axis, if necessary,
$U_+=\{(u,v)\,;\, v>0\}\ne\emptyset$ and
$U_-=\{(u,v)\,;\, v<0\}\ne\emptyset$.
In particular, $p$ is a type-changing point.
\end{proof}

Thus, characteristic curves
are also called {\it characteristic curves of type change}.

\begin{definition}\label{def:types-LPT}
Let $p\in \Sigma$ be a non-degenerate lightlike point,
$\gamma(t)$ $(|t|<\ep)$ a characteristic curve
passing through $p=\gamma(0)$,
and $\eta(t)$ a null vector field along $\gamma(t)$.
If $\gamma'(0)$ and $\eta(0)$ are linearly independent
(resp.\ linearly dependent),
we call $p$ 
the \emph{lightlike point of the first kind}
(resp.\ the \emph{lightlike point of the second kind\/}).
\footnote{%
In the case of wave fronts,
singular points of the first kind and the second kind 
were introduced in \cite{MSUY}.}%
Moreover, let $p\in LD$ be a lightlike point of the second kind.  
Then,
\begin{itemize}
\item
$p$ is said to be {\it admissible} if, 
for each open neighborhood $U$ of $p$,
the intersection $LD \cap U$ contains a lightlike point of the first kind.
\item
$p$ is said to be an \emph{$L_{\infty}$-point}
if $p$ is not admissible.
That is, there exists an open neighborhood $U$ of $p$
such that $LD \cap U$ consists of lightlike points of the second kind.
\end{itemize}
In other words,
$p$ is admissible
if there exists a sequence $\{p_n\}$ 
of lightlike points of the first kind such that 
$\lim_{n\to \infty} p_n = p$.
And, $p$ is an $L_{\infty}$-point
if there exists $\bar{\ep}>0$ such that
$\gamma(t)$ and $\eta(t)$ are linearly dependent
for all $|t|<\bar{\ep}$.
If we set
\begin{equation}\label{eq:character}
  \delta(t):= \det(\gamma'(t),\eta(t)),
\end{equation}
then $p=\gamma(0)$ is a lightlike point of the first kind
if and only if $\delta(0)\neq0$. 
For $k\geq 3$,
$p$ is called an {\it $L_k$-point\/} if 
$$
  \delta(0)= \dots = \delta^{(k-3)}(0)=0,\qquad
  \delta^{(k-2)}(0)\neq0.
$$
By definition, $L_k$-points are of the admissible second kind.
\end{definition}

\begin{remark}\label{rem:L-conelike}
We remark that lightlike points of the first kind
can be seen as `cuspidal-edge-like'.
More precisely, 
a germ of wave front $f:(\R^2,0)\to \R^3$
is locally diffeomorphic (or $\mathcal{A}$-equivalent)
to $f_{CE}:(\R^2,0)\to \R^3$
is called a cuspidal edge,
where $f_{CE}(u,v):=(u,v^2,v^3)$.
A useful criteria for recognition of the cuspidal edge
was given in \cite{KRSUY},
which implies the linear independence
of the `singular direction $\gamma'(0)$' and 
the 'null direction $\eta(0)$'
(for more details, see \cite{KRSUY}).
Hence, a lightlike point of the first kind 
has a singularity type
similar to the cuspidal edge.
Similarly, an $L_3$-point has a singularity type
similar to the swallowtail.
On the other hand,
a germ of wave front $f:(\R^2,0)\to \R^3$
is called a conelike singularity
if it is locally diffeomorphic to $f_C:(\R^2,0)\to \R^3$
defined by $f_C(u,v):=(v\cos u,v\sin u,v)$.
For a conelike singular point,
the singular direction $\gamma'(t)$ and 
the null direction $\eta(t)$ 
are linearly dependent
along the singular curve $\gamma(t)$.
Thus, an $L_{\infty}$-point has a singularity type
similar to the conelike singularity.
\end{remark}

Set the regular curve $\hat{\gamma}(t)$ in $M^3$ 
given by the image $\hat{\gamma}(t):=f(\gamma(t))$ 
of the characteristic curve $\gamma(t)$ through $f$.
We remark that
$p=\gamma(0)$ is of the first kind (resp.\ second kind)
if and only if
$\hat{\gamma}'(0)$ is spacelike (resp.\ lightlike).
If $p$ is of the first kind,
then there exists $\bar{\ep}>0$ such that,
for any $t\in (-\bar{\ep},\bar{\ep})$,
$\gamma(t)$ is also a lightlike point of the first kind.
Since $\gamma'(t)$ and $\eta(t)$ are linearly independent 
for each $t$, we have that
\begin{equation}\label{eq:image-is-spacelike}
  \hat{\gamma}(t)=f(\gamma(t))~ (|t|<\bar{\ep})~ 
  \text{is a spacelike regular curve in} ~M^3.
\end{equation}
On the other hand, $p$ is an $L_{\infty}$-point if and only if
there exists $\bar{\ep}>0$ such that
\begin{equation}\label{eq:image-is-lightlike}
  \hat{\gamma}(t)=f(\gamma(t))~ (|t|<\bar{\ep})~ 
  \text{is a lightlike regular curve in} ~M^3.
\end{equation}

In our terminology, 
\cite[Proposition 3.5]{HKKUY}
can be interpreted as follows.

\begin{fact}[{\cite[Proposition 3.5]{HKKUY}, \cite{UY_2018}}]
\label{fact:HKKUY}
Let $f:\Sigma\to M^3$ be a mixed type surface,
and let $p\in \Sigma$ be a non-degenerate lightlike point.
If the mean curvature $H$ is bounded on a neighborhood of $p$,
then $p$ must be an $L_{\infty}$-point.
\end{fact}

For the case of the Gaussian curvature,
see Theorem \ref{thm:introB}.
As an immediate consequence of this fact,
the mean curvature $H$ cannot be bounded
near a lightlike point of the first kind.

\begin{example}\label{ex:P-curved1}
We denote by $\R^3_1$ the Lorentz-Minkowski $3$-space
with the standard Lorentz metric
$
  \inner{\vect{x}}{\vect{x}} 
  =  x^2 + y^2 - z^2,
$
where $\vect{x}=(x,y,z)\in \R^3_1$.
Let $S^2$ be the unit sphere 
in the Euclidean $3$-space $\R^3$,
that is,
$S^2:=\{ (x,y,z)\in \R^3 \,;\, x^2+y^2+z^2=1\}$
(cf.\ Figure \ref{fig:S2}).
If we regard $S^2$ as a surface in $\R^3_1$,
$S^2$ is a mixed type surface.
Let us take the parametrization of $S^2$
as
\begin{equation}\label{eq:param-S2}
f(u,v)=
(\sin u \cos v,\cos u \cos v, \sin v)
\qquad
(u\in S^1,\,|v|<\pi/2),
\end{equation}
where $S^1:=\R/2\pi\Z$.
Since the induced metric $ds^2$ is written as
$
  ds^2=(\cos^2v)\,du^2 + (-\cos 2 v)\,dv^2,
$
we have
$\lambda=-\cos ^2v \cos 2 v$.
Hence, the lightlike point set is given by
$
LD=\{(u,v)\,;\, \cos 2v=0\} = \{(u,v)\,;\, v=\pm\pi/4\},
$
and every lightlike point of $S^2$ is non-degenerate.
Then, $\gamma(u)=(u,\pm\pi/4)$ is a characteristic curve.
Moreover, since $G=0$ on $\gamma$,
$\eta:=\partial_v$ gives a null vector field.
Therefore, every lightlike point of $S^2$ is of the first kind.
\end{example}

\begin{figure}[htb]
\begin{center}
  \resizebox{4cm}{!}{\includegraphics{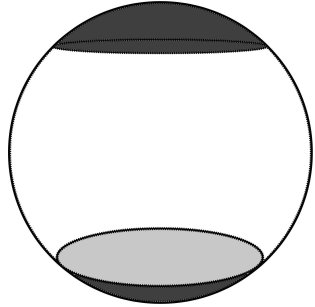}} 
\end{center}
\caption{
Regarding the unit sphere $S^2$ as a surface in $\R^3_1$,
both the 45th parallel north and south
consist of the lightlike points of the first kind.
The colored two domains are the set of spacelike points
and the other non-colored one is the set of timelike points.}
\label{fig:S2}
\end{figure}

In the rest of this section,
we give a characterization of 
lightlike points of the first kind.

\begin{proposition}\label{prop:type1}
Let $f:\Sigma\to M^3$ be a mixed type surface
and take a lightlike point $p\in LD$.
On a local coordinate neighborhood $(U;u,v)$ of $p$,
let $\lambda $ be the discriminant function
$\lambda = EG-F^2$.
Take a null vector field $\eta$ on a neighborhood of $p\in LD$.
Then, the following conditions 
are equivalent{\rm :}
\begin{itemize}
\item[$(P_1)$] $p$ is a lightlike point of the first kind,
\item[$(P_2)$] $\eta \lambda(p)\neq0$,
\item[$(P_3)$] $\eta_p \inner{\eta}{\eta}\neq0$,
\end{itemize}
where we set 
$\inner{\eta}{\eta}
:=ds^2(\eta,\eta)
=\inner{\eta f}{\eta f}$.
\end{proposition}

For the proof of this proposition,
we prepare the following lemma:

\begin{lemma}\label{lem:type1}
The conditions $(P_2)$ and $(P_3)$
in Proposition $\ref{prop:type1}$
do not depend on 
the choices of a null vector field $\eta$
and a local coordinate neighborhood $(U;u,v)$ of 
a lightlike point $p\in LD$.
\end{lemma}

\begin{proof}
Clearly, the condition $(P_2)$
is independent on the choice of $\eta$.
Take coordinate neighborhoods $(U;u,v)$, $(V;x,y)$
so that $U\cap V\ne\emptyset$. Then
\begin{align}
\label{eq:Exy}
  \hat{E} &= E u_x^2  +2F u_x v_x +G v_x^2,\\
\label{eq:Fxy}
  \hat{F} &= E u_x u_y  +F (u_x v_y+u_y v_x) +G v_xv_y,\\
\label{eq:Gxy}
  \hat{G} &= E u_y^2  +2F u_y v_y +G v_y^2
\end{align}
hold, where  
\begin{equation}\label{eq:1FF-hat}
  ds^2 = \hat{E}\, dx^2 + 2\hat{F}\, dx\,dy + \hat{G}\,dy^2.
\end{equation}
By a direct calculation, we have
$\hat{\lambda} = J^2 \lambda$,
where $\hat{\lambda} := \hat{E} \hat{G} - \hat{F}^2$
and $J:=u_x v_y - u_y v_x$.
Hence, at a lightlike point $p\in LD$,
we have
$
 \eta\hat{\lambda}(p) = J(p)^2 \eta\lambda(p),
$
which implies the condition $(P_2)$
does not depend on 
the choice of a local coordinate neighborhood 
$(U;u,v)$ of $p\in LD$.

With respect to the condition $(P_3)$,
let $\eta$, $\bar{\eta}$ be smooth vector fields
defined on a neighborhood $U$ of $p$
such that $\eta_p, \bar{\eta}_p \in \mathcal{N}_p$.
Then, there exist smooth functions $a$, $b$
and a smooth vector field $\xi$ on $U$ such that
$\xi$, $\eta$ are linearly independent,
$a(p)=0$, $b(p)\ne0$,
and
$
  \bar{\eta} = a\,\xi+b\,\eta.
$
Since
\begin{multline*}
 \bar{\eta}\inner{\bar{\eta}}{\bar{\eta}} 
 = \bar{\eta}(a^2)\inner{\xi}{\xi} + a^2\bar{\eta}\inner{\xi}{\xi}\\
   +2 \bar{\eta}(ab)\inner{\xi}{\eta} + 2ab\bar{\eta}\inner{\xi}{\eta}
   +\bar{\eta}(b^2)\inner{\eta}{\eta} + b^2\bar{\eta}\inner{\eta}{\eta},
\end{multline*}
we have
$
 \bar{\eta}_p\inner{\bar{\eta}}{\bar{\eta}}
 = b(p)^3 \eta_p\inner{\eta}{\eta}.
$
Hence, the condition $(P_3)$
is independent on the choice of $\eta$.
\end{proof}

\begin{proof}[Proof of Proposition \ref{prop:type1}]
First, we prove that the condition $(P_1)$ 
is equivalent to $(P_2)$.
Since both the conditions $(P_1)$ and $(P_2)$
implies that the lightlike point $p$ is non-degenerate,
we may assume that $p$ is a non-degenerate lightlike point.
Then,
$(\lambda_u,\lambda_v)\neq(0,0)$ holds at $p$. 
Let $\gamma(t)=(u(t),v(t))$ ($|t|<\ep$) be a 
characteristic curve passing through $p=\gamma(0)$ 
and $\eta=a(u,v)\partial_u+b(u,v)\partial_v$ 
a null vector field, 
where $a,b:U\to\R$ are smooth functions. 
Then $\lambda(\gamma(t))=0$ holds, and hence 
$$\left(\dfrac{d}{dt}\lambda(\gamma(t))=\right)\lambda(\gamma(t))'
=\lambda_u(\gamma(t))u'(t)+\lambda_v(\gamma(t))v'(t)=0$$
holds on $\gamma(t)$. 
Thus $\gamma'(t)=(u'(t),v'(t))$ is perpendicular to 
$(\lambda_u,\lambda_v)$ along $\gamma(t)$ 
in the sense of the standard Euclidean inner product of $\R^2$. 
On the other hand, $p$ is of the first kind if and only if 
$\det(\gamma',\eta)(0)\neq0$. 
Thus $p$ is a lightlike point of the first kind 
if and only if $\eta$ is not perpendicular to 
$(\lambda_u,\lambda_v)$ at $p$, 
namely, 
$\eta \lambda = a\lambda_u+b\lambda_v\neq0$
holds at $p$. 

Next, we prove the condition $(P_2)$ 
is equivalent to $(P_3)$.
By Lemma \ref{lem:type1},
we may take a coordinate neighborhood $(U;u,v)$ of $p$
such that $\partial_v$ is proportional to $\eta$.
Then, there exists a non-vanishing smooth function $\alpha(u,v)$
such that $\partial_v=\alpha\,\eta$.
Since $\partial_v$ is null at $p$,
we have $F(p)=G(p)=0$.
We remark that $E(p)>0$ holds (cf.\ \eqref{lem:corank1}).
By a direct calculation, we have
$
  \eta\lambda(p) = \alpha(p)^2 E(p) \eta_p\inner{\eta}{\eta},
$
which implies the desired result.
\end{proof}

\section{Invariants of lightlike points of the first kind}
\label{sec:invariant}

Let $f:\Sigma\to M^3$ be a mixed type surface.
A function $I : \Sigma\to \R$, or $I : LD\to \R$,
is called an {\it invariant},
that is, $I$ does not depend on the choice of 
coordinate system of the source.
An invariant $I : \Sigma\to \R$, or $I : LD\to \R$,
is {\it intrinsic}
if it can be expressed in terms of 
the first fundamental form $ds^2$.
An invariant $I : \Sigma\to \R$, or $I : LD\to \R$,
is {\it extrinsic}
if there exists a mixed type surface $\tilde{f}$ 
such that the first fundamental form of $\tilde{f}$
is the same as $f$,
but $I$ does not coincide.
In this section, we introduce several
invariants along the lightlike points of the first kind.

\subsection{Frames along lightlike points of the first kind}

Let $f:\Sigma\to M^3$ be a surface and 
$p\in \Sigma$ a lightlike point of the first kind.
On a coordinate neighborhood $(U;u,v)$ of $p$,
set $\lambda=EG-F^2$ as a discriminant function 
(cf.\ \eqref{eq:discriminant}).
Take a characteristic curve $\gamma(t)$ $(|t|<\ep)$
passing through $p=\gamma(0)$,
and a null vector field $\eta(t)$ along $\gamma(t)$.
If we take sufficiently small $\ep>0$,
we may assume
that $\gamma(t)$ consists of lightlike points of the first kind
(cf.\ \eqref{eq:image-is-spacelike}).
By \eqref{eq:image-is-spacelike},
\begin{equation}\label{eq:tangent}
  \vect{e}(t):= \frac{1}{ |\hat{\gamma}'(t)| } \hat{\gamma}'(t)
\end{equation}
is a spacelike tangent vector field of unit length.
Set the vector field $L(t)$ of $M^3$
along $\hat{\gamma}(t)$
as $L(t) := df(\eta(t))$.
Then, there exists a vector field $N(t)$ of $M^3$
along $\hat{\gamma}(t)$
such that
\begin{equation}\label{eq:L-normal}
  \inner{N(t)}{\vect{e}(t)}=0,\quad
  \inner{N(t)}{N(t)}=0,\quad
  \inner{N(t)}{L(t)}=1.
\end{equation}
We remark that 
such a vector field $N(t)$ is uniquely determined
by $\vect{e}(t)$ and $L(t)$.

\begin{lemma}\label{lem:normalized-null}
There exists a null vector field $\eta$
such that 
$\eta\inner{\eta f}{\eta f}=1$
holds along the characteristic curve.
{\rm (}We call such an $\eta$ 
a \emph{normalized null vector field}.{\rm )}
\end{lemma}

\begin{proof}
By Proposition \ref{prop:type1},
$\eta\inner{\eta f}{\eta f}$ does not vanish 
along the characteristic curve.
For any non-vanishing function $\alpha$,
$\bar{\eta}:=\alpha\,\eta$ 
is also a null vector field.
Differentiating 
$\inner{\bar{\eta}f}{\bar{\eta}f}
=\alpha^2\inner{\eta f}{\eta f}$,
we have
\begin{equation}\label{eq:vf-change}
  \bar{\eta}\inner{\bar{\eta}f}{\bar{\eta}f}
  =\alpha\,\eta(\alpha^2)\inner{\eta f}{\eta f}
  + \alpha^3\eta\inner{\eta f}{\eta f}.
\end{equation}
Setting
$\alpha:= (\eta\inner{\eta f}{\eta f})^{-1/3}$,
we obtain
$\bar{\eta}\inner{\bar{\eta}f}{\bar{\eta}f}|_{LD}=1$,
that is, $\bar{\eta}$ is the desired null vector field.
\end{proof}

A characteristic curve $\gamma(t)$ 
is said to be {\it parametrized by arclength} if
$|\hat{\gamma}'(t)|=1$ holds.
Then we have $\vect{e}(t)=\hat{\gamma}'(t)$.

\begin{definition}\label{def:invariants}
Let $p\in LD$ a lightlike point of the first kind.
Take a characteristic curve $\gamma(t)$ $(|t|<\ep)$
parametrized by arclength
such that $p=\gamma(0)$.
Also let $\eta$ be a normalized null vector field near $p$.
Set
\begin{align}
\label{eq:Lsingular-curvature-p}
  \kappa_L(p) 
  &:= \left.\inner{\nabla_{t}\vect{e}(t) }{ L(t) }\right|_{t=0},\\
\label{eq:Lnormal-curvature-p}
  \kappa_N(p) &:= \left.\inner{ \nabla_{t}\vect{e}(t) }{ N(t) }\right|_{t=0},\\ 
\label{eq:Lgeodesic-torsion-p}
  \kappa_G(p) &:= \left.\inner{L(t)}{\nabla_{t}N(t)}\right|_{t=0},
\end{align}
where $\nabla_t:=\nabla_{d/dt}$.
We call $\kappa_L(p)$
the {\it lightlike singular curvature} (or {\it L-singular curvature}),
$\kappa_N(p)$
the {\it lightlike normal curvature} (or {\it L-normal curvature}),
and
$\kappa_G(p)$
the {\it lightlike geodesic torsion} (or {\it L-geodesic torsion}) at $p$.
\end{definition}

For a characteristic curve $\gamma(t)$
which consists of lightlike points of the first kind,
we also describe the lightlike singular curvature $\kappa_L(\gamma(t))$
(resp.\ the lightlike normal curvature $\kappa_N(\gamma(t))$, 
the lightlike geodesic torsion $\kappa_G(\gamma(t))$)
along $\gamma(t)$,
as $\kappa_L(t)$
(resp.\ $\kappa_N(t)$, $\kappa_G(t)$),
unless otherwise noted.
It can be easily seen that 
the definitions of
$\kappa_L$, $\kappa_N$ and $\kappa_G$
are independent of the choice of the arclength parameter $t$.
Similarly, they are 
independent of the choice of
the normalized null vector field $\eta$,
which can be verified by the following.

\begin{lemma}
If both $\eta$ and $\bar{\eta}$
are normalized null vector field,
then $\eta|_{LD}=\bar{\eta}|_{LD}$ holds.
\end{lemma}

\begin{proof}
Let $\xi$ be a non-zero smooth vector field 
such that $\xi|_{LD}$ is tangent to $LD$.
Then, there exist smooth functions $\alpha$, $\beta$
such that $\alpha\ne0$, $\beta|_{LD}=0$ and
$\bar{\eta}=\alpha\,\eta + \beta\,\xi$.
It suffices to show that
$\alpha|_{LD}=1$ holds.
Since
\begin{align*}
  \bar{\eta}\inner{\bar{\eta} f}{\bar{\eta} f}
  &= \alpha\left( 
       \eta (\alpha^2) \inner{\eta f}{\eta f}
       +\alpha^2 \eta\inner{\eta f}{\eta f} 
       + 2\beta \eta(\alpha \inner{\eta f}{\xi f}) \right.\\
  &{}\hspace{4mm}     
       \left.+ 2\eta(\beta) \alpha \inner{\eta f}{\xi f}
       + 2\beta \eta(\beta) \inner{\xi f}{\xi f}
       + \beta^2 \eta(\inner{\xi f}{\xi f}) \right) 
        + \beta\,\xi (\inner{\bar{\eta} f}{\bar{\eta} f})
\end{align*}
holds, and evaluating this on $LD$,
we have $\alpha|_{LD}=1$.
\end{proof}

By a standard method, we have the following:

\begin{proposition}\label{prop:Frenet}
Let $\gamma(t)$ be the characteristic curve
consisting of lightlike points of the first kind
with the arclength parameter $t$,
and $\eta$ a normalized null vector field.
Then, the frame 
$\mathcal{F}(t):=(\vect{e}(t),L(t),N(t))$
satisfies 
$$
  {\nabla_{t}}\mathcal{F}(t)=\mathcal{F}(t) 
  \begin{pmatrix}
  0 & -\kappa_L & -\kappa_N\\
  \kappa_N & -\kappa_G & 0\\
  \kappa_L & 0 & \kappa_G
  \end{pmatrix}.
$$
\end{proposition}

In the case that 
$t$ is not necessarily an arclength parameter
and $\eta$ is not necessarily a normalized null vector field,
we have the following formula.

\begin{proposition}\label{prop:curvature-G}
Let $\gamma(t)$ be the characteristic curve
consisting of lightlike points of the first kind,
and $\eta$ a null vector field.
Then, 
$\kappa_L(t)$, $\kappa_N(t)$, $\kappa_G(t)$ 
are written as
\begin{align}
\label{eq:Lsingular-curvature-G}
  \kappa_L(t) 
  &= \frac{1}{ |\hat{\gamma}'(t)|^2 \beta(t)^{\frac1{3}} }
           \inner{ \nabla_{t} \hat{\gamma}'(t) }{ L(t) },\\
\label{eq:Lnormal-curvature-G}
  \kappa_N(t) 
  &= \frac{ \beta(t)^{\frac1{3}} }{ |\hat{\gamma}'(t)|^2 }
         \inner{\nabla_{t} \hat{\gamma}'(t)}{N(t)},\\
\label{eq:Lgeodesic-torsion-G}
  \kappa_G(t) 
  &= \frac{1}{ |\hat{\gamma}'(t)| }
       \left(
       \inner{L(t)}{\nabla_{t} N(t)} 
       + \frac1{3} (\log |\beta(t)|)'
       \right),
\end{align}
respectively,
where $\beta(t)$ is a non-zero function
defined by
$
  \beta(t) := \left.\eta \inner{\eta f}{\eta f} \right|_{\gamma(t)}.
$
\end{proposition}

\begin{proof}
As in Lemma \ref{lem:normalized-null},
$
  \bar{\eta}:= \beta^{-1/3} \eta
$
is a normalized null vector field.
Set $\bar{L}:=df(\bar{\eta})$, 
and set $\bar{N}$ so that 
$
  \inner{\hat{\gamma}'}{\bar{N}}
  =\inner{\bar{N}}{\bar{N}}=0$,
  $\inner{\bar{L}}{\bar{N}}=1
$
(cf.\ \eqref{eq:L-normal}).
Then, 
$\bar{L}(t) = \beta(t)^{-1/3} L(t)$,
$\bar{N}(t) = \beta(t)^{1/3} N(t)$
holds.
Moreover, 
$
  s:=\int_0^t |\hat{\gamma}'(\tau)|\,d\tau
$
gives an arclength parameter of $\gamma(t)$.
Then,
$$
  \kappa_L 
  = \inner{ \nabla_{ s } d\hat\gamma \left(\frac{d}{ds} \right) }{ \bar{L} },\qquad
  \kappa_N 
  = \inner{ \nabla_{ s } d\hat\gamma \left(\frac{d}{ds} \right) }{ \bar{N} } \quad 
 \left(\nabla_s:=\nabla_{\frac{d}{ds}}\right),
$$
implies \eqref{eq:Lsingular-curvature-G} 
and \eqref{eq:Lnormal-curvature-G}, respectively.
Moreover, 
\begin{equation*}
  \kappa_G
  = \inner{\bar{L}}{\nabla_{ s }\bar{N}}
  = \frac{ds}{dt}
       \inner{ \frac1{\beta^{\frac1{3}}} L }{
       \nabla_{ t } \left( \beta^{\frac1{3}} N \right) }
  = \frac{ 1 }{ |\hat{\gamma}'(t)| }\Bigl(
       \inner{ L }{\nabla_{ t }N } +(\log |\beta^{\frac1{3}}|)'  \Bigr)
\end{equation*}
holds, 
and hence,
we have \eqref{eq:Lgeodesic-torsion-G}.
\end{proof}

\begin{example}\label{ex:P-curved2}
As we seen in Example \ref{ex:P-curved1},
every lightlike point on the unit sphere $S^2$
is of the first kind.
Using Proposition \ref{prop:curvature-G},
we calculate their invariants 
$\kappa_L$, $\kappa_N$ and $\kappa_G$
here.
Let $f(u,v)$ be the parametrization of $S^2$ given by
\eqref{eq:param-S2}.
The characteristic curve is given by
$\gamma(t)=(t,\pm\pi/4)$,
and its image 
$\hat\gamma(t):=f(\gamma(t))$ 
is 
$\hat\gamma(t)=\frac{1}{\sqrt{2}} (\sin t,\cos t,\pm1).$
Since $\eta=\partial_v$ gives a null vector,
$L(t):=df_{\gamma(t)}(\eta)=\mp\frac{1}{\sqrt{2}} (\sin t,\cos t,\mp1)$
is a lightlike vector field along $\hat{\gamma}(t)$.
Then, 
the lightlike vector field along $\hat{\gamma}(t)$
satisfying \eqref{eq:L-normal}
is given by $N(t)=\mp\frac{1}{\sqrt{2}} (\sin t,\cos t,\pm1)$.
Applying Proposition \ref{prop:curvature-G}, we have
$$
\kappa_L(t)=\frac{1}{\sqrt[3]{2}},\qquad
\kappa_N(t)=\sqrt[3]{2},\qquad
\kappa_G(t)=0.
$$
In particular, the unit sphere has 
positive lightlike singular curvature 
$\kappa_L>0$ along $LD$
(cf.\ Corollary \ref{cor:shape}).
\end{example}

Let $\nabla$ be the Levi-Civita connection of $M^3$
and $\vect{e}(t)$ the unit tangent vector field 
of $\hat{\gamma}(t)$.
We call 
\begin{equation}\label{eq:curvature-vf}
R(t)=\nabla_{\vect{e}(t)}\vect{e}(t)
\end{equation}
the {\it curvature vector field} of $\hat{\gamma}(t)$.
To measure the causality of $R(t)$,
we set
$
  \theta(t) := \inner{R(t)}{R(t)}.
$
We call $\theta(t)$ 
the {\it causal curvature function} of $\hat{\gamma}(t)$.
By \eqref{eq:tangent},
The curvature vector field $R(t)$ is written as
$
  R(t)
  =( 
     - \inner{\hat{\gamma}'}{\nabla_{ t }\hat{\gamma}'}\hat{\gamma}'
     + \left| \hat{\gamma}' \right|^2 \nabla_{ t }\hat{\gamma}' 
  )/|\hat{\gamma}'|^4,
$
and hence we have
$
  \theta(t)
  =( 
    |\hat{\gamma}'(t)|^2 
    \inner{ \nabla_{ t }\hat{\gamma}'(t) }{ \nabla_{ t }\hat{\gamma}'(t) }  
     - \inner{ \hat{\gamma}'(t) }{ \nabla_{ t }\hat{\gamma}'(t) }^2
   )/|\hat{\gamma}'(t)|^6.
$

\begin{corollary}\label{cor:product}
The causal curvature function $\theta(t)$ satisfies
$\theta(t) = 2 \kappa_L(t)\kappa_N(t)$.
\end{corollary}

\begin{proof}
Suppose that 
$\hat{\gamma}(t)$ is parametrized by arc length
and $\eta$ is normalized.
Then $R(t)$ is written as $R(t)=\nabla_{ t }\hat{\gamma}'(t)$.
By Proposition \ref{prop:Frenet},
$R(t)=\kappa_L(t)N(t)+\kappa_N(t)L(t)$ holds.
Substituting this into $\theta(t)=\inner{R(t)}{R(t)}$,
we have $\theta(t) = 2 \kappa_L(t)\kappa_N(t)$.
\end{proof}

\subsection{Lightlike singular curvature is intrinsic}

Here, we show that the L-singular curvature $\kappa_L(t)$ 
depends only on the first fundamental form $ds^2$,
namely, $\kappa_L(t)$ is an intrinsic invariant.
To give an expression of $\kappa_L(t)$
on a coordinate neighborhood,
we introduce the following coordinate system.

\begin{definition}\label{def:adapted}
Let $p\in \Sigma$ be a lightlike point of the first kind.
Then, a coordinate system $(U;u,v)$ centered at $p$
is called {\it adapted} if
$\partial_v$ gives a null vector field
and the $u$-axis coincides with $LD$ on $U$.
\end{definition}

The existence of adapted coordinate systems 
can be proved easily.
Adapted coordinate systems can be determined intrinsically,
namely, they are defined in terms of the first fundamental form $ds^2$.
In fact, we can check that $(u,v)$ is an adapted coordinate system
if and only if 
\begin{equation}\label{eq:adapted-EFG}
  F(u,0)=G(u,0)=0
\end{equation}
holds,
cf.\ \eqref{eq:1st-FF}.
The following implies that
the lightlike singular curvature $\kappa_L$ 
is an intrinsic invariant.

\begin{proposition}\label{prop:kappa_L-adap}
Let $p\in \Sigma$ be a lightlike point of the first kind 
and take an adapted coordinate system $(U;u,v)$.
Then, we have
\begin{equation}\label{eq:kappa_L-adap}
  \kappa_L(u) 
  = - \frac{E_v(u,0)}{2E(u,0)\sqrt[3]{G_v(u,0)} }.
\end{equation}
\end{proposition}

\begin{proof}
Since
$
  \inner{\nabla_u f_u}{f_v} 
  = \inner{f_{u}}{f_v}_u -\inner{\nabla_v f_u}{f_u},
$
we have
\begin{equation}\label{eq:fuufv}
  \inner{\nabla_u f_u}{f_v} 
  = F_u- \frac1{2}E_v.
\end{equation}
As $\gamma(u)=(u,0)$ is a characteristic curve
and $\eta=\partial_v$ is a null vector field,
\eqref{eq:Lsingular-curvature-G} implies
$\kappa_L(u)  
= \left. \inner{ \nabla_u f_u }{ f_v}/ 
   (E \sqrt[3]{ G_v } ) \right|_{v=0}.
$
Since $F_u(u,0)=0$ holds by \eqref{eq:adapted-EFG},
we have \eqref{eq:kappa_L-adap}.
\end{proof}

\subsection{Balancing curvature}

Here, we introduce an invariant called 
the \emph{balancing curvature} $\kappa_B$ 
(Definition \ref{def:kappa_B})
along the lightlike points of the first kind,
which is related to the behavior of 
the Gaussian curvature, as we shall see later
in Section \ref{sec:Gauss}.
To define the balancing curvature,
we use a coordinate system
which is specially customized 
at lightlike points of the first kind.

\begin{definition}
Let $p\in \Sigma$ be a lightlike point of the first kind.
Then an adapted coordinate neighborhood $(U;u,v)$
of $p$ is called \emph{specially adapted}
if 
$E(u,0)=1$ and $\lambda_v(u,0)=1$
hold, cf.\ \eqref{eq:1st-FF}.
\end{definition}

We remark that,
on an adapted coordinate neighborhood $(U;u,v)$
such that $E(u,0)=1$,
the condition
$\lambda_v(u,0)=1$ is equivalent to 
$G_v(u,0)=1$.

\begin{proposition}\label{prop:special-adapted}
Let $p\in \Sigma$ be a lightlike point of the first kind.
Then there exists a specially adapted coordinate system
centered at $p$.
\end{proposition}

\begin{proof}
Let $(z,y)$ be an adapted coordinate system
and set 
$x(z):=\int_{0}^{z} E^*(t,0)^{1/2}\,dt$,
where $E^*:=\inner{f_z}{f_z}$.
Then $(x,y)$
is an adapted coordinate system
such that $\hat{E}(x,0)=1$ 
(cf.\ \eqref{eq:1FF-hat}).
Setting $u:=x$ and
$v:=\hat{G}_y(x,0)^{1/3}\,y$,
we have 
\begin{equation}\label{eq:dv}
  u_x=1,\quad
  u_y=0,\quad
  v_x= \left(\hat{G}_y(x,0)^{1/3} \right)_{x}\, y,\quad
  v_y= \hat{G}_y(x,0)^{1/3}.
\end{equation}
Then \eqref{eq:Exy}, \eqref{eq:Fxy} and \eqref{eq:Gxy}
yield that
$E(u,0)=1$, $F(u,0)=G(u,0)=0$.
Namely, $(u,v)$ is also an adapted coordinate system
such that $E(u,0)=1$.
Thus, it suffices to show that $\lambda_v(u,0)=1$.
By \eqref{eq:Gxy}, we have $\hat{G}=Gv_y^2$.
Thus, on the $v$-axis, we have
$
  G_v(u,0) = v_y(x,0)^{-3}\hat{G}_y(x,0).
$
Hence, \eqref{eq:dv} yields $\lambda_v(u,0)=G_v(u,0)=1$.
\end{proof}


Now, we shall introduce an invariant called 
the \emph{balancing curvature} 
along the lightlike points of the first kind.

\begin{definition}\label{def:kappa_B}
Let $p\in \Sigma$ be a lightlike point of the first kind.
For a specially adapted coordinate neighborhood $(U;u,v)$
of $p$, we set
\begin{equation}\label{eq:kappa_B}
  \kappa_B(p) 
  := \left.\left( 
        -\frac{1}{2}E_{vv} +F_{uv} 
        + \frac{E_v}{10}\left( G_{vv}-2\left(F_v \right)^2 \right)
      \right)\right|_{u=v=0}.
\end{equation}
We call $\kappa_B(p)$ the {\it balancing curvature} at $p=(0,0)$
(cf.\ \eqref{eq:1st-FF}).
\end{definition}

The balancing curvature 
is related to the behavior of the geodesic curvature
and the Gaussian curvature at lightlike points
(cf.\ Remark \ref{rem:geodesic-curvature}
and Theorem \ref{thm:introB}).

Let $\gamma(t)$ be a characteristic curve
which consists of lightlike points of the first kind.
Then, we also describe the balancing curvature along $\gamma(t)$,
$\kappa_B(\gamma(t))$, as $\kappa_B(t)$, unless otherwise noted. 
Since 
the $u$-axis gives a characteristic curve, $\gamma(u)=(u,0)$,
on a specially adapted coordinate neighborhood $(U;u,v)$,
we have
\begin{equation}\label{eq:kappa_B_axis}
  \kappa_B(u) 
  = \left.\left( 
        -\frac{1}{2}E_{vv} +F_{uv} 
        + \frac{E_v}{10}\left( G_{vv}-2\left(F_v \right)^2 \right)
      \right)\right|_{v=0}.
\end{equation}

The definition of $\kappa_B$
does not depend on a choice of 
specially adapted coordinate systems
(Proposition \ref{prop:kappa_B}).
To prove it, we prepare the following.

\begin{lemma}\label{lem:s-adapted-rule}
Let $(u,v)$ and $(x,y)$ be specially adapted coordinate systems 
centered at a lightlike point $p$ of the first kind.
Then, we have
$$
  u_x(x,0)=\pm1,\quad
  u_y(x,0)=v_x(x,0)=0,\quad
  v_y(x,0)=1.
$$
\end{lemma}

\begin{proof}
The metric coefficients
$\hat{E}$, $\hat{F}$, $\hat{G}$
with respect to the coordinate system $(x,y)$
are expressed as
\eqref{eq:Exy}, 
\eqref{eq:Fxy}, 
\eqref{eq:Gxy},
respectively,
cf.\ \eqref{eq:1FF-hat}. 
The adaptedness of $(u,v)$ and $(x,y)$ 
yields $v(x,0)=0$, and hence $v_x(x,0)=0$ holds.
Since $(u,v)$ and $(x,y)$ are specially adapted,
$E(u,0)=G_v(u,0)=\hat{E}(x,0)=\hat{G}_y(x,0)=1$ and
$F(u,0)=G(u,0)=\hat{F}(x,0)=\hat{G}(x,0)=0$ hold.
Hence we have 
$
  u_x(x,0)^2=1,~
  u_y(x,0)^2=0,~
  v_y(x,0)^3=1,
$
which gives the desired conclusion.
\end{proof}

\begin{proposition}\label{prop:kappa_B}
The definition of the balancing curvature $\kappa_B$
does not depend on a choice of 
specially adapted coordinate systems.
\end{proposition}

\begin{proof}
Let $(u,v)$ and $(x,y)$ be specially adapted coordinate systems 
centered at a lightlike point $p$ of the first kind.
By \eqref{eq:Exy}, \eqref{eq:Fxy}, \eqref{eq:Gxy},
and Lemma \ref{lem:s-adapted-rule},
we have 
$\hat{E}_{y}(x,0) = E_v(u,0)$,
$\hat{F}_{y}(x,0) = \pm (F_v(u,0) + u_{yy}(x,0))$
and
{\allowdisplaybreaks
\begin{align*}
  \hat{E}_{yy}(x,0) &= E_{vv}(u,0) + E_v(u,0) v_{yy}(x,0) \pm 2 u_{xyy}(x,0),\\
  \hat{F}_{xy}(x,0) &= F_{uv}(u,0) \pm u_{xyy}(x,0),\\
  \hat{G}_{yy}(x,0) &= G_{vv}(u,0) + 4 F_v(u,0) u_{yy}(x,0) 
                            + 2 u_{yy}(x,0)^2 + 5 v_{yy}(x,0),
\end{align*}}%
where the sign ``$\pm$'' corresponds to 
$u_x(x,0)=\pm1$ in Lemma \ref{lem:s-adapted-rule}.
Then, we can check that
\begin{multline*}
  \left.\left( \frac{\hat{E}_{yy}}{2} -\hat{F}_{xy} 
      - \frac{\hat{E}_y( \hat{G}_{yy}-2\hat{F}_y^2)}{10}\right)\right|_{y=0}
  =\left.\left( \frac{E_{vv}}{2} -F_{uv} 
      - \frac{E_v( G_{vv}-2F_v^2 )}{10}\right)\right|_{v=0},
\end{multline*}
which gives the desired result.
\end{proof}

For an adapted coordinate system,
which is not necessarily special,
we have the following formula of the balancing curvature.

\begin{proposition}\label{prop:kappa_B_adapted}
Let $p\in \Sigma$ be a lightlike point of the first kind
and $(U;u,v)$ an adapted coordinate neighborhood
centered at $p=(0,0)$.
Then, the balancing curvature is written as 
\begin{multline}\label{eq:kappa_B_adapted}
  \kappa_B(p) =  \left. 
  \frac{ -5G_v \left(E E_{vv} - 2 E F_{uv} + E_u F_v\right)
         +E_v \left(E G_{vv} - 2 (F_v)^2\right)}
       {10E^2 (G_v)^{\frac{5}{3}}} \right|_{u=v=0}.
\end{multline}
\end{proposition}

\begin{proof}
Let $(x,y)$ be a specially adapted coordinate system.
Set $\hat{E}$, $\hat{F}$, $\hat{G}$
as \eqref{eq:1FF-hat}.
The adaptedness of $(u,v)$ and $(x,y)$ 
yields $v(x,0)=0$, and hence
$v_x(x,0)=0$
holds.
By \eqref{eq:Exy}, \eqref{eq:Fxy}, \eqref{eq:Gxy},
we have 
$u_x(x,0) = \pm 1/\sqrt{E(u,0)}$
and
$u_y(x,0) = 0$.
Then, 
\begin{equation}\label{eq:dell-xy}
  (\partial_x)_{(x,0)} = u_x(x,0)(\partial_u)_{(u,0)},\qquad
  (\partial_y)_{(x,0)} = v_y(x,0)(\partial_v)_{(u,0)}
\end{equation}
holds.
Taking the partial derivative of \eqref{eq:Gxy} 
with respect to $y$ along the $x$-axis,
we have $1=G_v(u,0)v_y(x,0)^3$.
Namely, 
$v_y(x,0) = 1/\sqrt[3]{G_v(u,0)}$
holds.
Taking the differentiation of 
$u_x(x,0)^2=1/E(u,0)$ 
with respect to $x$,
we have
$u_{xx}(x,0) = -E_u(u,0)/(2E(u,0)^2)$,
where we used \eqref{eq:dell-xy}.
Therefore,
\begin{align*}
  &-\frac{1}{2}\hat{E}_{yy}(x,0) +\hat{F}_{xy}(x,0) 
      + \frac{1}{10}\hat{E}_y(x,0)\left( \hat{G}_{yy}(x,0)-2\hat{F}_y(x,0)^2 \right)\\
  &= \frac{-1}{2E^2 G_v^{5/3}}
  \Biggl(G_v \left(E E_{vv} - 2 E F_{uv} 
  + E_u F_v\right)
  \left.
  -\frac1{5}E_v \left(E G_{vv} - 2 F_v^2\right)\Biggr)
  \right|_{(u,0)}
\end{align*}
holds, which gives the desired result.
\end{proof}

\begin{remark}[Behavior of the geodesic curvature]
\label{rem:geodesic-curvature}
Let $p\in \Sigma$ be a lightlike point of the first kind.
Take a specially adapted coordinate neighborhood $(U;u,v)$ of $p$.
For each $s$, with sufficiently small $|s|$,
let $\{\gamma^s(u)\}$ be the family of 
the $u$-curves $\gamma^s(u):=(u,s)$.
Then 
$\vect{e}^s(u):= (1/\sqrt{E}) \partial_u |_{v=s}$
gives the unit tangent vector field along $\gamma^s(u)$.
For $s\ne0$,
the unit conormal vector field along $\gamma^s(u)$
is given by
$
  \vect{n}^s(u)
  :=(1/\sqrt{E|\lambda|}) (-F\partial_u+E\partial_v)
  |_{v=s}.
$
Then,
the geodesic curvature function $\kappa_g^s(u)$
of $\gamma^s(u)$ is given by
$\kappa_g^s(u)
:=\inner{\nabla_{\vect{e}^s}\vect{e}^s }{\vect{n}^s}$.
By a direct calculation, we have
$$
  \kappa_g^s(u)
  =\left.
  \frac{-F E_u+2EF_u-EE_v}{2|\lambda|^{1/2}E^{3/2}}
  \right|_{v=s}.
$$
Thus, 
$\tilde{\kappa}_g(u,s):=|\lambda(u,s)|^{1/2}\kappa_g^s(u)$
is a smooth function with respect to the variables $(u,s)$
which has the following expansion:
$
\tilde{\kappa}_g(u,s)
  = k_1(u) + k_2(u)\, s + k_3(u,s)\,s^2,
$
where
\begin{align*}
  &k_1(u)
  =\tilde{\kappa}_g(u,0)
  = -\frac{1}{2}E_v(u,0),\\
  \quad
  &k_2(u)
  = \left.\frac{\partial}{\partial s}\right|_{s=0} \tilde{\kappa}_g(u,s)
  = -\frac{5}{4}E_v(u,0)^2-\frac1{2}E_{vv}(u,0)+F_{uv}(u,0),
\end{align*}
and $k_3(u,s)$ is a smooth function.
Since the L-singular curvature $\kappa_L(u)$ 
is written as
$\kappa_L(u)=-(1/2)E_v(u,0)$,
we have that $k_1(u)$ coincides with $\kappa_L(u)$.
Hence, if $\kappa_L(u)$ vanishes along the characteristic curve 
$\gamma(u)=(u,0)$,
the limit $\lim_{s\to 0}\kappa_g^s(u)$ exists,
and it is identically zero along $\gamma(u)$.
In this case,
the characteristic curve $\gamma(u)$ may be regarded 
as a kind of geodesics, which is 
related to the notion of {\it pseudo-geodesics} \cite{Steller},
cf.\ Lemma \ref{lem:Prop1-Steller}.
Moreover, in this case, 
$k_2(u)$ coincides with the balancing curvature $\kappa_B(u)$.
Hence, we may conclude that
the L-singular curvature $\kappa_L$
and the balancing curvature $\kappa_B$
are related to the behavior of the geodesic curvature
at lightlike points.
\end{remark}

\section{Behavior of the invariants
at lightlike points of the second kind}
\label{sec:2nd-kind}

In this section, 
we study the behavior of the four invariants,
$\kappa_L$, 
$\kappa_N$, 
$\kappa_G$, 
$\kappa_B$,
at lightlike points of the admissible second kind.
Since we investigate the local properties here, 
we assume that the ambient manifold $M^3$ is oriented
throughout this section.
We first consider $\kappa_L$
and prove the assertion (i) 
of Theorem \ref{thm:introA}.
Next, to calculate $\kappa_N$ and $\kappa_G$,
we prepare a formula of the cross product 
at a lightlike point
(Lemma \ref{lem:gaiseki}).
Then, we give a proof of the assertion (ii) 
of Theorem \ref{thm:introA}.
Finally, we show a similar result about 
$\kappa_G$ and $\kappa_B$ 
in Theorems \ref{thm:kG} and Theorem \ref{thm:kB}.

\subsection{Behavior of the lightlike singular curvature}

We shall prove that the L-singular curvature $\kappa_L$
diverges to $-\infty$ at a lightlike point of the admissible second kind.
We carry out the calculation on 
the following coordinate system:

\begin{definition}
Let $p\in \Sigma$ be a lightlike point of the admissible second kind.
A local coordinate system $(U;u,v)$
centered at $p$
such that the $u$-axis gives
a characteristic curve $\gamma(u)=(u,0)$
is said to be a {\it characteristic coordinate system}.
Then, there exists a smooth function $\ep(u)$ on the $u$-axis
such that 
\begin{equation}\label{eq:2nd-null}
  \eta(u)=\partial_u+\ep(u)\partial_v,
\end{equation}
is a null vector field $\eta(u)$ along $\gamma(u)=(u,0)$.
The admissibility of $p$ implies that
$\ep(0)=0$, 
$\ep(u)\not\equiv 0.$
Namely,
\begin{equation}\label{eq:zeroset}
  Z^c:=\{(u,0)\in U \,;\, \ep(u)\ne0\} 
\end{equation}
is not empty, and
consists of lightlike points of the first kind.
\end{definition}

Since $L(u)=df(\eta(u))=f_u(u,0)+\ep(u)f_v(u,0)$
is a lightlike vector field along 
$\hat{\gamma}(u)=f(u,0)$, 
we have
$
  \inner{L(u)}{f_u(u,0)}=\inner{L(u)}{f_v(u,0)}=0.
$
That is, 
\begin{equation}\label{eq:EFG-u-axis}
  E(u,0) 
  =\ep(u)^2 G(u,0),\quad
  F(u,0)= - \ep(u) G(u,0)
  \quad(G(u,0)>0).
\end{equation}

\begin{lemma}\label{lem:lambda_v}
On the $u$-axis, 
$\inner{\nabla_u f_{u}}{f_u}=\ep (2\ep' G +\ep G_u)/2$
and
$\inner{\nabla_u f_{u}}{f_v}=-E_v/2 -\ep' G - \ep G_u$
hold.
Moreover, it holds that
$E_v(0,0)\neq0$.
\end{lemma}

\begin{proof}
Substituting \eqref{eq:EFG-u-axis}
into 
$\inner{\nabla_u f_{u}}{f_u} = E_u/2$
and 
$\inner{\nabla_u f_u}{f_v}= F_u - E_v/2$
(cf.\ \eqref{eq:fuufv}),
we obtain the first assertion.
Differentiating $\lambda=EG-F^2$, we have
$\lambda_v(u,0)=E_vG+\ep^2 GG_v+2\ep GF_v.$
Since $\lambda(u,0)=0$,
the non-degeneracy yields 
$\lambda_v(0,0)\neq0$,
which implies the second assertion.
\end{proof}

\begin{proof}[Proof of the assertion {\rm (i)} of Theorem \ref{thm:introA}]
Let $p$ be a lightlike point of the admissible second kind.
Take a characteristic coordinate system $(U;u,v)$ centered at $p$.
Since $p$ is a lightlike point of the admissible second kind,
there exists a sequence $\{p_n\}_{n\in \N} \subset Z^c$
such that $\lim_{n\to \infty}p_n=p$.
We may write $p_n=(u_n,0)\,(\in Z^c)$.
By Proposition \ref{prop:curvature-G}, we have
$
\kappa_L(p_n)= 
       \inner{ \nabla_u f_{u} }{ \eta f }/
              (\inner{f_u}{f_u} \beta(u)^{1/3}) |_{(u,v)=(u_n,0)},
$
where
$\beta(u) = \eta\inner{\eta f}{\eta f}|_{\gamma(u)}$. 
Since
$
  \inner{\eta f}{\eta f}
  = E + 2 \ep F + \ep^2 G
$
holds
and $\inner{\eta f}{\eta f}$ is identically zero along the $u$-axis,
we have
\begin{equation}\label{eq:beta}
  \beta(u)
  = \ep(u)\, \rho(u)\qquad
  \left(
  \rho(u) := E_v(u,0) +2\ep(u) F_v(u,0) + \ep(u)^2 G_v(u,0)
  \right).
\end{equation}
On the other hand,
Lemma \ref{lem:lambda_v} yields
$
  \inner{ \nabla_u f_{u} }{ \eta f }
  = -\ep \left( \ep G_u + E_v \right)/2
$
on $Z^c$.
Together with \eqref{eq:EFG-u-axis},
we have
$$
  \kappa_L(p_n)
  = -\frac{1}{2} \left. \frac{ E_v + \ep G_u }
      { \ep^{4/3} 
      G \rho(u)^{1/3}} 
      \right|_{(u,v)=(u_n,0)}.
$$
Hence,
$$
  \lim_{n\to \infty}\ep(u_n)^{4/3} \kappa_L(p_n)
  = -\frac{ E_v(0,0)^{2/3} }{2G(0,0)} \,(<0)
$$
holds.
Therefore, we have that
$\kappa_L(p_n)$ diverges to $-\infty$ as $n\to \infty$.
\end{proof}

\subsection{Cross product}
Let $Q$ be a $2$-dimensional 
degenerate subspace of $T_pM^3$ $(p\in M^3)$.
Denote by $\mathcal{L}(Q)$ the union of 
the set of lightlike vectors of $Q$ and the zero-vector $\vect{0}\in Q$,
which is a $1$-dimensional degenerate subspace of $Q$.

\begin{lemma}\label{lem:gaiseki}
Take a non-zero spacelike tangent vector $\vect{v}\in T_pM^3$. 
Let $\vect{w}\in T_pM^3$ be a lightlike tangent vector 
satisfying $\inner{\vect{v}}{\vect{w}}=0$.
Then, either
$$
  \vect{v} \times \vect{w} = |\vect{v}| \, \vect{w}
  \quad\text{or}\quad
  \vect{v} \times \vect{w} = -|\vect{v}| \, \vect{w}
$$
holds.
Moreover, if $Q$ is a $2$-dimensional 
degenerate subspace of $T_pM^3$,
then we have $\vect{x} \times \vect{y} \in \mathcal{L}(Q)$
for each $\vect{x}$, $\vect{y}\in Q$.
\end{lemma}

\begin{proof}
Set a spacelike unit tangent vector $\vect{e}_1:= |\vect{v}|^{-1}\vect{v}$.
Take $\vect{e}_2$, $\vect{e}_3\in T_pM^3$
so that $\left\{ \vect{e}_1, \vect{e}_2, \vect{e}_3 \right\}$
forms a positively oriented orthonormal basis for $T_pM^3$
(cf.\ \eqref{eq:ONB}).
Set two lightlike vectors $\vect{e}_+, \vect{e}_- \in T_pM^3$ as
$\vect{e}_+ := \vect{e}_2- \vect{e}_3$ and
$\vect{e}_- := \vect{e}_2+ \vect{e}_3$,
respectively.
Since $\vect{w} \in T_pM^3$ is a lightlike vector which is orthogonal to $\vect{v}$,
we have that $\vect{w}$ is parallel to
either $\vect{e}_+$ or $\vect{e}_-$.
Hence, there exists a non-zero real number $b\in \R$ such that
either $\vect{w} = b\vect{e}_+$ or
$\vect{w} = b\vect{e}_-$ holds.
By a straightforward calculation (cf.\ \eqref{eq:CROSS}),
$\vect{v} \times \vect{w} = |\vect{v}| \, \vect{w}$ 
(resp.\ $\vect{v} \times \vect{w} = -|\vect{v}| \, \vect{w}$)
holds if
$\vect{w} = b\vect{e}_+$
(resp.\ $\vect{w} = b\vect{e}_-$).

On the other hand, 
take a non-zero spacelike tangent vector $\vect{v}\in Q$
and let $\vect{w}\in Q$ be a lightlike tangent vector.
Then, 
we have
$Q={\rm Span}(\vect{v}, \vect{w})$,
$\mathcal{L}(Q)={\rm Span}(\vect{w})$, and
$\inner{\vect{v}}{\vect{w}}=0$.
Since $\vect{x}$ and $\vect{y}$
are written as linear combinations of 
$\vect{v}$ and $\vect{w}$,
there exists a square matrix $A$ of order $2$
such that 
$
  (\vect{x}, \vect{y})
  =(\vect{v}, \vect{w})A
$
holds.
Then, 
we have 
$
  \vect{x}\times \vect{y}
  =(\det A)  \vect{v}\times \vect{w}
  =\pm (\det A)  |\vect{v}|\, \vect{w} \in \mathcal{L}(Q),
$
which proves the assertion.
\end{proof}

Let $p\in \Sigma$ be a lightlike point of the admissible second kind.
Take a characteristic coordinate system $(U;u,v)$
centered at $p$.
Since 
$f_v(u,0)$ is a non-zero spacelike vector field (cf.\ \eqref{eq:2nd-null})
and 
$L(u)=df(\eta(u))$
is a lightlike vector field which is orthogonal to $f_v(u,0)$,
we may apply Lemma \ref{lem:gaiseki}.
Then, there exists $\sigma=\pm1$ such that 
$$
  f_v(u,0) \times L(u) = -\sigma \sqrt{G(u,0)}\,L(u)
$$
along the $u$-axis.
Substituting $L(u) = f_u(u,0)+\ep(u)f_v(u,0)$ into this,
we have
$
  f_u \times f_v = \sigma \sqrt{G}\,(f_u+\ep f_v)
$
along the $u$-axis.
Therefore, 
there exists a smooth 
vector field $\psi$ of $M^3$
along $f:U \to M^3$
such that
\begin{equation}\label{eq:psi-characteristic}
  f_u \times f_v = \sigma \sqrt{G}\,(f_u+\ep f_v) + v\, \psi
\end{equation}
holds on $U$.

\begin{lemma}\label{lem:kNG}
For $(u,0)\in Z^c$,
we have
\begin{align}
  \label{eq:kN-char}
  \kappa_N(u)
  &= \frac{\rho(u)^{1/3}}{\ep(u)^{8/3}G}
    \left(- \frac{\ep'(u)}{2} 
     - \frac{E_v(E_v + 2\ep'(u) G)}{4 \rho(u) G}
     + \ep(u) \nu_1(u)\right),\\
  \label{eq:kG-char}
  \kappa_G(u)
  &=  \frac{{\rm sgn}(\ep(u))}{\ep(u)^2\sqrt{G}}
  \left( 
  -\frac{\ep'(u)}{6}
     + \frac{E_v(E_v + \ep'(u) G)}{2 \rho(u) G}
    +\ep(u)g_1(u)
   \right),
\end{align}
where $E_v=E_v(u,0)$, $G=G(u,0)$,
and 
\begin{equation}\label{eq:rho}
  \rho(u) := E_v(u,0) +2\ep(u) F_v(u,0) + \ep(u)^2 G_v(u,0).
\end{equation}
\end{lemma}

\begin{proof}
Let $N(u)$ be the vector field $N(u)$ of $M^3$
along $\hat{\gamma}(u):=f(u,0)$
such that 
\begin{equation}\label{eq:N(u)}
  \inner{N(u)}{\hat{\gamma}'(u)}
  =\inner{N(u)}{N(u)}=0,\quad
  \inner{N(u)}{L(u)}=1
\end{equation}
holds on $Z^c$
(cf.\ \eqref{eq:L-normal}).
First, we calculate $N(u)$.

Applying the division lemma\footnote{%
For example, see \cite[Appendix A]{UY_geloma}.} 
to \eqref{eq:EFG-u-axis},
there exists smooth functions $\tilde{E}$, $\tilde{F}$
such that 
$E=\ep^2G+v\,\tilde{E}$ and
$F=-\ep\,G+v\,\tilde{F}$
on a neighborhood of $p=(0,0)$.
Differentiating these identities with respect to $v$,
we have
\begin{align}\label{eq:tilde-E}
  &\tilde{E}(u,0)= E_v(u,0) - \ep(u)^2 G_v(u,0),\\
\label{eq:tilde-F}
  &\tilde{F}(u,0)= F_v(u,0) + \ep(u) G_v(u,0).
\end{align}
By \eqref{eq:psi-characteristic},
we have 
$v\inner{f_u}{\psi}=-\sigma \sqrt{G} (E+\ep F)
=-\sigma v \sqrt{G} (\tilde{E} + \ep \tilde{F})$.
Hence, 
$
\inner{f_u}{\psi}  = -\sigma \sqrt{G} (\tilde{E} + \ep \tilde{F})
$
holds on $U$.
Similarly, we have
$\inner{f_v}{\psi}  = -\sigma \sqrt{G}\tilde{F}$ and
$\inner{\psi}{\psi}  = (\tilde{F})^2$.
By \eqref{eq:tilde-E} and \eqref{eq:tilde-F}, it follows that 
\begin{align}
\label{eq:fu-psi}
  &\inner{f_u}{\psi} = -\sigma \sqrt{G} (E_v+\ep F_v),\\
\label{eq:fv-psi}
  &\inner{f_v}{\psi} = -\sigma \sqrt{G} (F_v+\ep G_v),\\
\label{eq:psi-psi}
  &\inner{\psi}{\psi} = (F_v+\ep G_v)^2
\end{align}
along the $u$-axis.

Set
$N(u) = c_1(u) f_u(u,0) + c_2(u) f_v(u,0) +c_3(u) \psi(u,0)$.
Substituting this into \eqref{eq:N(u)},
we have
\begin{gather*}
  \sigma(E_v+\ep F_v) c_3 = \ep \sqrt{G}(\ep c_1 -c_2),\qquad
  -c_3 \sigma \sqrt{G} \rho=1\\
  0= G(\ep c_1 -c_2)^2 
     -2 \sigma \sqrt{G}(c_1(E_v
     +\ep F_v)+c_2(F_v+\ep G_v))c_3
     +(F_v+\ep G_v)^2 (c_3)^2,
\end{gather*}
where we used 
\eqref{eq:fu-psi}, \eqref{eq:fv-psi} and \eqref{eq:psi-psi}.
Solving this, we obtain
\begin{equation}\label{eq:N-chara}
  N(u) 
  = -\frac{1}{2\ep^2G} f_u
     + \frac{E_v-\ep^2 G_v}{2\ep \rho G} f_v
     - \frac{1}{\sigma\rho\sqrt{G}} \psi,
\end{equation}
where the right hand side is evaluated on $Z^c$
and $\rho(u)$ is the smooth function on the $u$-axis
defined by \eqref{eq:rho}.

By Proposition \ref{prop:curvature-G}, we have
$
  \kappa_N(u)
  = \beta(u)^{1/3}
  \inner{\nabla_u f_u(u,0)}{N(u)}/|f_u(u,0)|^2,
$
where
$\beta(u) = \eta\inner{\eta f}{\eta f}|_{\gamma(u)}
=\ep(u)\, \rho(u)$ (cf.\ \eqref{eq:beta}), 
and $\rho(u)$ is the function defined by
\eqref{eq:rho}.
By \eqref{eq:N-chara} and Lemma \ref{lem:lambda_v}, 
we have
{\allowdisplaybreaks
\begin{align}
\nonumber
  \ep(u)\inner{\nabla_u f_u(u,0)}{N(u)}
  &= -\biggl(\frac{2\ep' G + \ep G_u}{4G}
     + \ep\frac{\inner{\nabla_u f_u}{\psi}}{\sigma\rho\sqrt{G}}\\
\nonumber
  &\hspace{1.8cm}
     + \left.\frac{(E_v - \ep^2 G_v)(E_v
     +2\ep' G +2\ep G_u )}{4 \rho G}
     \biggr)\right|_{v=0}\\
\label{eq:fuu-N}
  &= - \frac{\ep'(u)}{2} 
     - \frac{E_v(E_v + 2\ep'(u) G)}{4 \rho(u) G}
     + \ep(u) \nu_1(u),
\end{align}}%
where $E_v=E_v(u,0)$, $G=G(u,0)$, 
and $\nu_1(u)$ is a smooth function defined on the $u$-axis.
By \eqref{eq:EFG-u-axis},
$
  |f_u(u,0)|^2 
  = E(u,0)=\ep(u)^2G(u,0)
$
holds.
Thus, we have
$$
  \ep(u)^{8/3}\kappa_N(u)
  = \frac{\rho(u)^{1/3}}{G(u,0)}
    \left(- \frac{\ep'(u)}{2} 
     - \frac{E_v(E_v + 2\ep'(u) G)}{4 \rho(u) G}
     + \ep(u) \nu_1(u)\right),
$$
which implies \eqref{eq:kN-char}.

With respect to $\kappa_G$,
we have
$
  \kappa_G(u)
  = \left( \inner{L(u)}{\nabla_u N(u)} 
         +\beta'(u)/(3\beta(u)) 
   \right)/|\hat{\gamma}'(u)|,
$
by Proposition \ref{prop:curvature-G}.
Since $\nabla_u L(u) = \nabla_u f_u +\ep' \,f_v +\ep\nabla_u f_v$
holds along the $u$-axis,
we have
\begin{multline}\label{eq:num-kG}
  \inner{\nabla_u L(u)}{N(u)}
  = \inner{\nabla_u f_u(u,0)}{N(u)} \\
   +\ep'(u) \,\inner{f_v(u,0)}{N(u)} 
   +\ep(u)\inner{\nabla_u f_v(u,0)}{N(u)}.
\end{multline}
By \eqref{eq:N(u)},
we obtain
\begin{equation}\label{eq:num-kG-2}
  \ep(u)\inner{f_v(u,0)}{N(u)}
  = 1.
\end{equation}
Now, by \eqref{eq:N-chara} and Lemma \ref{lem:lambda_v}
\begin{align}
\nonumber
  \ep(u)^2\inner{\nabla_u f_v(u,0)}{N(u)}
  &= \left.\left(-\frac{E_v}{4G}
    +\ep
      \frac{G_u(E_v-\ep^2 G_v) 
            - 4\sigma\sqrt{G}\inner{\nabla_u f_v}{\psi}}{4\rho G}
    \right)\right|_{v=0}\\
\label{eq:fuv-N}
  &= -\frac{E_v(u,0)}{4G(u,0)}
    -\ep(u)\nu_2(u)
\end{align}
holds,
where $\nu_2(u)$ is a smooth function defined on the $u$-axis.
Using 
$
  \inner{L(u)}{\nabla_u N(u)}=-\inner{\nabla_u L(u)}{N(u)},
$
and substituting \eqref{eq:fuu-N},
\eqref{eq:num-kG-2},
\eqref{eq:fuv-N}
into \eqref{eq:num-kG},
we have
\begin{multline*}
  \ep(u)\inner{L(u)}{\nabla_u N(u)}
  = -\frac{\ep'(u)}{2}
     + \frac{E_v(E_v+\rho(u) 
     + 2\ep'(u) G)}{4 \rho(u) G}\\
    +\ep(u)(\nu_2(u)-\nu_1(u)).
\end{multline*}
Since
$
  \ep\beta'/\beta
  = \ep'+\ep\rho'/\rho
$
and $\rho(u)=E_v + \ep(2F_v+\ep G_v)$,
we have
$$
  \ep|\ep|\kappa_G(u)
  = \frac{1}{\sqrt{G}}
  \left( 
  -\frac{\ep'(u)}{6}
     + \frac{E_v(E_v + \ep'(u) G)}{2 \rho(u) G}
    +\ep(u)g_1(u)
   \right),
$$
where $g_1(u)$ is a smooth function defined on the $u$-axis.
\end{proof}

\begin{proof}[Proof of the assertion {\rm (ii)} of Theorem \ref{thm:introA}]
As in the proof of the assertion {\rm (i)} of Theorem \ref{thm:introA},
take a characteristic coordinate system $(U;u,v)$ 
centered at $p$,
and a sequence $\{p_n\}_{n\in \N} \subset Z^c$
such that $\lim_{n\to \infty}p_n=p$.
We may write $p_n=(u_n,0)\,(\in Z^c)$.
By Lemma \ref{lem:kNG}, we have
$\ep(u)^{8/3}\kappa_N(u) = r_N(u)$,
where we set
$$
  r_N(u):=\left.
    \frac{\rho(u)^{1/3}}{G}
    \left(- \frac{\ep'(u)}{2} 
     - \frac{E_v(E_v + 2\ep'(u) G)}{4 \rho(u) G}
     + \ep(u) \nu_1(u)\right)\right|_{v=0}.
$$
Since $\ep(0)=0$ and $\rho(0)=E_v(0,0)$ by \eqref{eq:rho}, 
it holds that 
$$
  r_N(0)
  = -\frac{E_v(0,0)^{4/3}}{4 G(0,0)^2} - \ep'(0)\frac{E_v(0,0)^{1/3}}{G(0,0)}.
$$
First, assume that $p$ is not an $L_3$-point.
Then, $\ep'(0)=0$ holds.
Since $u_n$ converges to $0$ as $n\to \infty$,
we have
$$
  \lim_{n\to \infty} \ep(u_n)^{8/3}\kappa_N(p_n)
  = r_N(0)
  = -\frac{E_v(0,0)^{4/3}}{4 G(0,0)^2}.
$$
Hence, $\kappa_N(p_n)$ diverges to $-\infty$ as $n\to \infty$.
Next, let us assume that $p$ is an $L_3$-point.
Since $\ep(0)=0$ and $\ep'(0)\ne0$ hold,
there exist a smooth function $\ep_0(u)$
such that
$\ep(u) = u\, \ep_0(u)$ $(\ep_0(0)\ne0)$ holds.
On the other hand, 
since $r_N(u)$ is a smooth function on the $u$-axis,
there exist an integer $k\in \Z$ and 
a smooth function $\nu_0(u)$
such that $r_N(u) = u^k \nu_0(u)$ holds.
Then, on $Z^c$, we have
$$
  \kappa_N(u_n)
  = u^{k-\frac8{3}} \frac{\nu_0(u_n)}{\sqrt[3]{\ep_0(u_n)^8}}.
$$
Since $k-\frac8{3}$ is not an integer,
$\kappa_N(u_n)$ tends to $0$ or diverges as $n\to \infty$,
which gives the desired result.
\end{proof}

Similarly, we have the following result 
for the lightlike geodesic torsion $\kappa_G$.

\begin{theorem}\label{thm:kG}
Let $f:\Sigma\to M^3$ be a mixed type surface
in a Lorentzian $3$-manifold $M^3$
and $p\in \Sigma$ a lightlike point 
of the admissible second kind.
If $p$ is not an $L_3$-point, then
the lightlike geodesic torsion $\kappa_G$ 
is unbounded at $p$.
\end{theorem}

\begin{proof}
Take a characteristic coordinate system $(U;u,v)$ 
centered at $p$,
and a sequence $\{p_n\}_{n\in \N} \subset Z^c$
such that $\lim_{n\to \infty}p_n=p$.
We may write $p_n=(u_n,0)\,(\in Z^c)$.
By Lemma \ref{lem:kNG}, we have
$\ep(u)|\ep(u)|\kappa_G(u) = r_G(u)$,
where we set
$$
  r_G(u):=\left.
    \frac{1}{\sqrt{G}}
  \left(  -\frac{\ep'(u)}{6}
     + \frac{E_v(E_v + \ep'(u) G)}{2 \rho(u) G}
    +\ep(u)g_1(u)
   \right)\right|_{v=0}.
$$
Since 
$$
  r_G(0)=\frac{1}{\sqrt{G(0,0)}}
   \left( \frac{E_v(0,0)}{2 G(0,0)}+ \frac{1}{3}\ep'(0)
   \right),
$$
if $p$ is not an $L_3$-point, 
namely, if $\ep'(0)=0$,
then
$$
  \lim_{n\to \infty} \ep(u_n)|\ep(u_n)|\kappa_G(p_n) 
  = \frac{E_v(0,0)}{2 G(0,0)^{3/2}} \,(\ne0),
$$
which gives the desired result.
\end{proof}

As a corollary of Theorem \ref{thm:introA},
we have the following.

\begin{corollary}
Let $p\in \Sigma$ be 
a lightlike point of the second kind,
and let $\gamma(t)$ $(|t|<\ep)$ be 
a characteristic curve passing through $p=\gamma(0)$.
Assume that 
$p$ is not an $L_3$-point,
and 
$\gamma(t)$ is of the first kind for each $t\ne0$.
Then, the curvature vector field $R(t)$
{\rm (}cf.\ \eqref{eq:curvature-vf}{\rm )}
of $\hat{\gamma}(t)=f\circ \gamma(t)$
is spacelike for sufficiently small $t\ne0$.
\end{corollary}

\begin{proof}
Let $\theta(t)$ be
the causal curvature function $\theta(t):=\inner{R(t)}{R(t)}$.
Corollary \ref{cor:product} yields
$\theta(t)=2\kappa_L(t)\kappa_N(t)$.
Together with Theorem \ref{thm:introA},
we have that $\theta(t)$ diverges to $+\infty$ at $0$.
Hence, $\theta(t)>0$ for sufficiently small $t\ne0$.
\end{proof}

\subsection{Behavior of the balancing curvature}

Finally, we investigate the behavior of 
the balancing curvature $\kappa_B$ at 
a lightlike point of the admissible second kind.

Let $p\in \Sigma$ be a lightlike point of the admissible second kind.
Take a characteristic coordinate system $(U;u,v)$
centered at $p$.
The set $Z^c$ defined by \eqref{eq:zeroset}
consists of lightlike points of the first kind.
Fix a point $p_0=(u_0,0)\in Z^c$ such that 
$|u_0|$ is sufficiently small.
Since $\ep(u_0)>0$, 
\begin{equation}\label{eq:XY}
  x:=u-\frac1{\ep(u)}v,\qquad
  y:=v
\end{equation}
defines a new coordinate on a neighborhood of $p_0$.

\begin{lemma}\label{lem:char-adapt}
The coordinates $(x,y)$ given in \eqref{eq:XY}
is an adapted coordinate system 
{\rm (}cf.\ Definition {\rm \ref{def:adapted})}.
Moreover, for a smooth function $h=h(u,v)$,
we have 
\begin{align}
\label{eq:Hy}
  &h_y(u,0)= \frac1{\ep} h_u + h_v\\
\label{eq:Hxy}
  &h_{xy}(u,0)= -\frac{\ep'}{\ep^2}h_u +\frac1{\ep}h_{uu}+ h_{uv}\\
\label{eq:Hyy}
  &h_{yy}(u,0)= -\frac{2\ep'}{\ep^3}h_u 
             +\frac1{\ep^2}h_{uu} +\frac2{\ep}h_{uv}+h_{vv},
\end{align}
where the right hand sides are evaluated on $Z^c$.
\end{lemma}

\begin{proof}
We set
$$
 \Delta := 1+\frac{\ep'}{\ep^{2}} v.
$$
Since
$
  x_u = \Delta,~
  x_v = -1/\ep,~
  y_u=0,~
  y_v=1,
$
we have
$
  u_x(u,v) = 1/\Delta,~
  u_y(u,v) = 1/(\ep\Delta),~
  v_x(u,v)=0,~
  v_y(u,v)=1,
$
and hence,
$
\partial_x = (1/{\Delta})\partial_u,~
\partial_y = (1/{\ep\Delta})\partial_u+\partial_v
$
holds.
On the $u$-axis 
$\partial_x = \partial_u$,
$\partial_y = \ep^{-1}\partial_u + \partial_v$
holds.
Thus, we have that
the $x$-axis is the singular set, 
and 
$\partial_y$ is in the null direction
on the the singular set,
in particular, $(x,y)$ is an adapted coordinate system.
By a direct calculation,
we have the formulas \eqref{eq:Hy}, \eqref{eq:Hxy} and \eqref{eq:Hyy}. 
\end{proof}

Let $\hat{E}$, $\hat{F}$, $\hat{G}$
be the functions as \eqref{eq:1FF-hat}.
By Proposition \ref{prop:kappa_B_adapted},
the balancing curvature $\kappa_B(u)$
is written as
\begin{equation}\label{eq:kB-def}
  \kappa_B(u)
  =   \frac{-1}{2\hat{E}^2 (\hat{G}_y)^{\frac{5}{3}}}
        \left( A_1 -\frac1{5}A_2\right)
\end{equation}
for each $(u,0)\in Z^c$,
where we set
\begin{align}
\label{eq:A1}
  &A_1(u):= \hat{G}_y \left(\hat{E} \hat{E}_{yy} 
         - 2 \hat{E}\hat{F}_{xy} 
        + \hat{E}_x \hat{F}_y\right),\\
\label{eq:A2}
  &A_2(u):= \hat{E}_y \left(\hat{E} \hat{G}_{yy} 
  - 2 (\hat{F}_y)^2\right).
\end{align}

\begin{lemma}\label{lem:kB-prepare}
Set $b_0(u):=\ep'(u) G(u,0) -E_v(u,0)$.
Then, we have
{\allowdisplaybreaks
\begin{align}
\label{eq:order12}
  \hat{E}(u,0)&=\ep^2 G,\qquad
  \hat{E}_x(u,0)=\ep^2 G_u + 2\ep \ep' G,\\
\label{eq:order-10}
  \hat{E}_y(u,0) &=E_u + \ep G_u,\quad
  \hat{F}_y(u,0)=\frac1{\ep}\left( -b_0(u) +\ep F_v\right).
\end{align}
Moreover,} there exist smooth functions $\tau_i(u)$ 
on the $u$-axis $(i=1,2,3,4)$,
such that
{\allowdisplaybreaks
\begin{align}
\label{eq:order-2}
  \hat{E}_{yy}(u,0) 
  &= \frac1{\ep^2}\left( 4\ep'b_0(u) + \ep \tau_1(u) \right),\\
\label{eq:order-2b}
  \hat{F}_{xy}(u,0) 
  &= \frac1{\ep^2}\left( \ep'b_0(u) + \ep \tau_2(u) \right),\\
\label{eq:order-2c}
  \hat{G}_{y}(u,0) 
  &= \frac1{\ep^2}\left( E_v + \ep \tau_3(u) \right),\\
\label{eq:order-4}
  \hat{G}_{yy}(u,0) 
  &= \frac1{\ep^4}\left( 2\ep'(\ep' G -4E_v) + \ep \tau_4(u) \right)
\end{align}
holds.}
\end{lemma}

\begin{proof}
By a direct calculation using Lemma \ref{lem:char-adapt},
we have
{\allowdisplaybreaks
\begin{align*}
& \Delta_y(u,0) = \frac{\ep'}{\ep^2},\quad
 \Delta_{yy}(u,0) = 2\frac{\ep'\ep''-2(\ep')^2}{\ep^4},\\
& \ep_{y}(u,0) = \frac{\ep'}{\ep},\quad
 \ep_{yy}(u,0) = 2\frac{\ep'\ep''-2(\ep')^2}{\ep^3},\\
& E_{y}(u,0) = E_v + 2 \ep' G + \ep G_u,\quad
 F_{y}(u,0) = -\frac{\ep'}{\ep}G -G_u +F_v,\\
& G_{y}(u,0) = \frac{1}{\ep}G_u +G_v,\quad
 G_{yy}(u,0) = -\frac{2\ep'}{\ep^3}G_u 
             +\frac1{\ep^2}G_{uu} +\frac2{\ep}G_{uv}+G_{vv},
\end{align*}
where we used \eqref{eq:EFG-u-axis}.
Moreover, since
\begin{align*} 
E_{yy}(u,0) 
&=  \frac1{\ep^2}( 2(\ep \ep''+(\ep')^2)G
    + 4\ep\ep'G_u+\ep^2G_{uu} ) \\
&\hspace{3cm}  
    -\frac{2\ep'}{\ep^2}(2\ep'G + \ep G_u)
    +\frac2{\ep}E_{uv}+E_{vv},
\end{align*}
there exists} 
a smooth function $q_1(u)$ defined on the $u$-axis
such that
$
E_{yy}(u,0)= (-2(\ep')^2G+\ep q_1(u))/\ep^2
$
holds.
Similarly, 
there exists 
a smooth function $q_2(u)$ defined on the $u$-axis
such that
$
F_{yy}(u,0)= (2(\ep')^2G+\ep q_2(u))/\ep^3.
$
By \eqref{eq:Exy}, \eqref{eq:Fxy}, \eqref{eq:Gxy}, 
we have{\allowdisplaybreaks
$$
  \hat{E}(u,v)=\frac1{\Delta^2} E,\quad
  \hat{F}(u,v)
  =\frac1{\Delta}\left( \frac1{\ep\Delta}E+F \right) ,\quad
  \hat{G}(u,v)
  =\frac1{\ep^2\Delta^2}E+\frac2{\ep\Delta}F+G.
$$
Since} $\hat{E}(u,0)=E(u,0)$,
and $\hat{E}_x(u,0)=E_u(u,0)$,
we have \eqref{eq:order12}.
Next we consider $\hat{E}_y$ and $\hat{F}_y$.
Substituting 
the above $\Delta_y$, $\ep_y$, $E_y$, $F_y$ into 
$\hat{E}_y = (-2\Delta_y E + \Delta E_y)/\Delta^3$ and
$$
  \hat{F}_y 
  = 
  -\frac{\ep_y \Delta+2\ep \Delta_y }{\ep^2\Delta^3} E 
      + \frac{1}{\ep\Delta^2}E_y -\frac{\Delta_y}{\Delta^2}F + \frac{1}{\Delta}F_y,
$$
we have \eqref{eq:order-10}.
With respect to $\hat{E}_{yy}$,
substituting 
the above $\Delta_y$, $\Delta_{yy}$, $E_y$, $E_{yy}$ into 
$$
  \hat{E}_{yy}
  = -2\frac{\Delta \Delta_{yy}-3\Delta_y^2}{\Delta^4}E
  -4 \frac{\Delta_y}{\Delta^2}E_y
  +\frac1{\Delta^2}E_{yy},
$$
we have
$
  \hat{E}_{yy}(u,0)
  = \left(
    4\ep'(\ep'G-E_v) + \ep (q_1 -4\ep''G-4\ep' G_u)
  \right)/\ep^2.
$
Thus, setting 
$\tau_1(u):=q_1(u) -4\ep''(u)G(u,0)-4\ep'(u) G_u(u,0)$,
\eqref{eq:order-2} holds.
By a similar calculation,
we have 
\eqref{eq:order-2b}, 
\eqref{eq:order-2c}
and \eqref{eq:order-4}.
\end{proof}

\begin{lemma}\label{lem:kB-char}
For $(u,0)\in Z^c$,
we have
\begin{equation}\label{eq:kB-char}
  \kappa_B(u) = 
  -\frac{\ep(u)^{-8/3}}{5 G^2 (E_v + \ep \tau_3)^{5/3}}
  \left( (E_v)^2(2\ep'(u) G+E_v) +\ep(u)\tau_0(u)  \right),
\end{equation}
where $\tau_i(u)$ $(i=0,3)$ 
are smooth functions on the $u$-axis,
and $E_v=E_v(u,0)$, $G=G(u,0)$.
\end{lemma}

\begin{proof}
By Lemma \ref{lem:kB-prepare},
we have 
$$
  \ep(u)^2\left(A_1(u)-\frac1{5}A_2(u)\right)
  = \frac{2}{5}\left((E_v)^2(2\ep'(u) G+E_v) +\ep(u)\tau_0(u) \right),
$$
where $A_i(u)$ $(i=1,2)$ are the functions
defined by \eqref{eq:A1}, \eqref{eq:A2},
and
$\tau_0(u)$ is a smooth function defined on the $u$-axis.
Moreover, 
by Lemma \ref{lem:kB-prepare},
$
  \ep(u)^{-2/3}\hat{E}^2\hat{G}_y^{5/3}
  = G^2(E_v+\ep(u)\tau_3(u))^{5/3}
$
holds along the $u$-axis,
where $\tau_3(u)$ is a smooth function as in Lemma \ref{lem:kB-prepare}.
Substituting these identities into \eqref{eq:kB-def},
we have \eqref{eq:kB-char}.
\end{proof}

\begin{theorem}\label{thm:kB}
Let $f:\Sigma\to M^3$ be a mixed type surface
in a Lorentzian $3$-manifold $M^3$
and $p\in \Sigma$ a lightlike point 
of the admissible second kind.
Then, the balancing curvature $\kappa_B$ 
converges to $0$ or diverges to $\pm\infty$ at $p$.
In particular, if $p$ is not an $L_3$-point,
then $\kappa_B$ tends to $-\infty$ at $p$.
\end{theorem}

\begin{proof}
As in the proof of Theorem \ref{thm:introA},
take a characteristic coordinate system $(U;u,v)$ 
centered at $p$,
and a sequence $\{p_n\}_{n\in \N} \subset Z^c$
such that $\lim_{n\to \infty}p_n=p$.
We may write $p_n=(u_n,0)\,(\in Z^c)$.
By Lemma \ref{lem:kNG}, we have
$\ep(u)^{8/3}\kappa_B(u) = r_B(u)$,
where we set
$$
  r_B(u):=\left.
    -\frac{1}{5 G^2 (E_v + \ep \tau_3)^{5/3}}
    \left( (E_v)^2(2\ep'(u) G+E_v) +\ep(u)\tau_0(u)  \right)
    \right|_{v=0}.
$$
Since $\ep(0)=0$, it holds that 
$$
  r_B(0)
  = -\frac{E_v(0,0)^{4/3}}{5 G(0,0)^2} - \ep'(0)\frac{2E_v(0,0)^{1/3}}{5G(0,0)}.
$$
First, assume that $p$ is not an $L_3$-point.
Then, $\ep'(0)=0$ holds.
Since $u_n$ converges to $0$ as $n\to \infty$,
we have
$$
  \lim_{n\to \infty} \ep(u_n)^{8/3}\kappa_B(p_n)
  = r_B(0)
  = -\frac{E_v(0,0)^{4/3}}{5 G(0,0)^2}.
$$
Hence, $\kappa_B(p_n)$ diverges to $-\infty$ as $n\to \infty$.
Next, let us assume that $p$ is an $L_3$-point.
Since $\ep(0)=0$ and $\ep'(0)\ne0$ hold,
there exist a smooth function $\ep_0(u)$
such that
$\ep(u) = u\, \ep_0(u)$ $(\ep_0(0)\ne0)$ holds.
On the other hand, 
since $r_B(u)$ is a smooth function on the $u$-axis,
there exist an integer $k\in \Z$ and 
a smooth function $\mu_0(u)$
such that $r_B(u) = u^k \mu_0(u)$ holds.
Then, on $Z^c$, we have
$$
  \kappa_B(u_n)
  = u^{k-\frac8{3}} \frac{\mu_0(u_n)}{\sqrt[3]{\ep_0(u_n)^8}}.
$$
Since $k-\frac8{3}$ is not an integer,
$\kappa_B(u_n)$ tends to $0$ or diverges as $n\to \infty$,
which gives the desired result.
\end{proof}

\section{Behavior of Gaussian curvature}
\label{sec:Gauss}

In this section, 
we study the behavior of the Gaussian curvature $K$
at non-degenerate lightlike points.
As in the previous section,
since we investigate the local properties, 
we assume that the ambient manifold $M^3$ is oriented
throughout this section.
After calculating $K$ at a lightlike point of the first kind 
(Proposition \ref{prop:lambda-K}),
we give a characterization of the boundedness of $K$ 
in terms of the invariants $\kappa_L$, $\kappa_N$ and $\kappa_B$
in Theorem \ref{thm:boundedness}.
Then, we show a relationship between 
$\kappa_L$ and the shape of the surface
(Corollary \ref{cor:shape}),
and prove Theorem \ref{thm:introB}
and Corollary \ref{cor:introC}.

\subsection{The lightlike normal curvature on a coordinate neighborhood}

We have discussed a property of 
the cross product at lightlike points
in Lemma \ref{lem:gaiseki}.
Then we have the following.

\begin{lemma}\label{lem:psi}
Let $f : \Sigma \to M^3$ be a mixed type surface
and $p$ a non-degenerate lightlike point.
Take a coordinate neighborhood $(V;u,v)$ of $p$ 
such that $\partial_v$ gives a null vector field.
Then, there exist an open neighborhood 
$U\subset V$ of $p$,
and a smooth vector field 
$\psi$ of $M^3$ along $f$ 
defined on $U$ such that 
$($cf.\ \eqref{eq:1st-FF}$)$
either 
$$
  f_u\times f_v
  = \sqrt{E} f_v+\lambda\, \psi
  \quad\text{or}\quad
  f_u\times f_v
  = -\sqrt{E} f_v+\lambda\, \psi
$$
holds on $U$.
\end{lemma}

\begin{proof}
For each $q\in LD\cap V$, we have that
$f_u(q)\in T_{f(q)}M^3$ is a spacelike vector and 
$f_v(q) \in T_{f(q)}M^3$ is a lightlike vector such that
$\inner{f_u(q)}{f_v(q)}=0$.
Then, by Lemma \ref{lem:gaiseki}, 
either 
$$
f_u(q) \times f_v(q) =  \sqrt{E(q)} \, f_v(q)
\quad\text{or}\quad
f_u(q) \times f_v(q) = - \sqrt{E(q)} \, f_v(q)
$$
holds.
Since $q\in LD\cap V$ is chosen arbitrarily, 
we may assume that either
$$
  f_u \times f_v- \sqrt{E} \, f_v=0
  \quad\text{or}\quad
  f_u \times f_v+ \sqrt{E} \, f_v=0  
$$
holds on $LD\cap V$,
without loss of generality.
Since $d\lambda\ne0$,
the division lemma 
(cf.\ \cite[Appendix A]{UY_geloma}) 
implies that
there exist an open neighborhood $U\subset V$ of $p$,
and a smooth vector field $\psi$ of $M^3$  
along $f$ defined on $U$ such that
either
$$
  f_u \times f_v- \sqrt{E} \, f_v=\lambda\, \psi
  \quad\text{or}\quad
  f_u \times f_v+ \sqrt{E} \, f_v=\lambda\, \psi
$$
holds on $U$.
Thus, we have the desired result.
\end{proof}

\begin{definition}\label{def:admissible-coord}
Let $p$ be a non-degenerate lightlike point.
A coordinate neighborhood $(U;u,v)$ of $p$
is called {\it admissible} if there exists 
a smooth vector field 
$\psi$ of $M^3$ along $f$ 
defined on $U$
such that 
\begin{equation}\label{eq:psi-pm}
  f_u\times f_v
  = \sigma\sqrt{E} f_v+\lambda\, \psi
\end{equation}
holds on $U$, where $\sigma=\pm1$.
\end{definition}

\begin{remark}\label{rem:admissible}
According to Lemma \ref{lem:gaiseki},
it holds that, if $(U;u,v)$ is admissible, 
then $\partial_v$ gives a null vector field on $U$.
Hence, together with Lemma \ref{lem:psi},
we may identify 
an admissible coordinate system
with 
a coordinate system such that $\partial_v$ is a null vector field,
on a sufficiently small neighborhood 
of a non-degenerate lightlike point.
In particular, 
admissible coordinate systems can be determined intrinsically.

In the case that $p$ is of the first kind,
any adapted (or specially adapted) coordinate system 
$(U;u,v)$ centered at $p$ satisfies that 
$\partial_v$ is a null vector field.
Therefore, taking a smaller neighborhood if necessary,
we may assume that 
any adapted (or specially adapted) coordinate system
is admissible.
\end{remark}

The L-normal curvature $\kappa_N$ 
can be expressed using the map $\psi$.
Set functions $\phi_{ij}$ $(i,j=1,2)$ as
\begin{equation}\label{eq:aij}
  \phi_{11}:=\inner{\nabla_u f_{u}}{\psi},\quad
  \phi_{12}:=\inner{\nabla_u f_{v}}{\psi}=\inner{\nabla_v f_{u}}{\psi},\quad
  \phi_{22}:=\inner{\nabla_v f_{v}}{\psi}.
\end{equation}

\begin{lemma}\label{lem:kappa-N-sp}
On a specially adapted coordinate system $(U;u,v)$
centered at a lightlike point $p$ of the first kind,
the L-normal curvature $\kappa_N(u)$ 
along the $u$-axis is given by
\begin{equation}\label{eq:kappa-N-sp}
  \kappa_N(u) = -\sigma\, \phi_{11}(u,0).
\end{equation}
\end{lemma}

\begin{proof}
Since $(u,v)$ is adapted,
$\gamma(u):=(u,0)$ is a characteristic curve 
and $\eta:=\partial_v$ is a null vector field.
Set $\hat{\gamma}(u):= (f \circ \gamma)(u)=f(u,0)$.
Since 
$\inner{\hat{\gamma}'(u)}{\hat{\gamma}'(u)}=E(u,0)=1$,
the characteristic curve
$\gamma(u)$ is parametrized by arclength,
cf.\ \eqref{eq:1st-FF}.
Moreover, since 
$
  1=G_v
  =\partial_v\inner{f_v}{f_v}
  =\eta\inner{\eta f}{\eta f}
$
holds along the $u$-axis, 
$\eta  =\partial_v$ is a normalized null vector field.
Then, the L-normal curvature $\kappa_N(u)$ 
is given by 
\begin{equation}\label{eq:Lnormal-sp0}
  \kappa_N(u)=\inner{\nabla_u f_{u}(u,0)}{N(u)},
\end{equation}
where $N(u)$ is a vector field along 
$\hat{\gamma}(u)$
satisfying \eqref{eq:L-normal}.

Since $F(u,0)=G(u,0)=0$, 
there exist smooth functions 
$\hat{F},\hat{G}$ on a neighborhood of $p$ such that 
$F=v\,\hat{F}$ and
$G=v\,\hat{G}$. 
Then, on the $u$-axis, we have 
\begin{equation}\label{eq:FG-hat-v}
  \hat{F}(u,0) = F_v(u,0),\qquad
  \hat{G}(u,0)=1,
\end{equation}
where we used $G_v(u,0)=1$.
Using 
$\inner{f_u\times f_v}{f_u\times f_v} = -EG+F^2$,
we have
$\inner{f_u}{\psi} = - \sigma \sqrt{E} F/\lambda$,
$\inner{f_v}{\psi} = - \sigma \sqrt{E} G/\lambda$
and
$\inner{\psi}{\psi} = F^2/\lambda^2$
on $\{(u,v) \in U \,;\, v\ne0\}$.
Since 
\begin{equation}\label{eq:lambda-sp}
  \lambda = EG-F^2 = v(E \hat{G} -v \hat{F}^2),
\end{equation}
it follows that
$\inner{f_u}{\psi} 
= - \sigma \sqrt{E} \hat{F}/(E \hat{G} -v \hat{F}^2)$,
$\inner{f_v}{\psi} 
  = - \sigma\sqrt{E} \hat{G}/(E \hat{G} -v \hat{F}^2)$
and
$\inner{\psi}{\psi} 
  = \hat{F}^2/(E \hat{G} -v \hat{F}^2)^2$
on $U$, by the continuity.
In particular, by \eqref{eq:FG-hat-v},
\begin{equation}\label{eq:psi-inner}
  \inner{f_u}{\psi} = - \sigma F_v ,\quad
  \inner{f_v}{\psi} = - \sigma ,\quad
  \inner{\psi}{\psi} = F_v^2
\end{equation}
holds on the $u$-axis.
Then, 
$
  N(u) := -\sigma \psi(u,0) - F_v(u,0) f_u(u,0)
$
satisfies \eqref{eq:L-normal},
where we used \eqref{eq:psi-inner}.
Substituting $N(u)$ into \eqref{eq:Lnormal-sp0},
we have \eqref{eq:kappa-N-sp}.
Here, we used
$\inner{f_u(u,0)}{\nabla_u f_{u}(u,0)}=E_u(u,0)/2=0$.
\end{proof}

\subsection{Boundedness of Gaussian curvature}

Here, we calculate the Gaussian curvature 
near non-degenerate lightlike points.

Let $f:\Sigma\to M^3$ be a mixed type surface.
Take a coordinate neighborhood $(U;u,v)$.
Since $\inner{f_u \times f_v}{f_u \times f_v}
=-\lambda$, 
we have that
$\nu:=(f_u\times f_v)/\sqrt{|\lambda|}$
gives a unit normal vector field along $f$
on the non-lightlike point set $U_+\cup U_-$,
where we denote by $U_+$ (resp.\ $U_+$) 
the set of spacelike (resp.\ timelike) points on $U$.
We set the smooth functions $h_{11}$, $h_{12}$ and $h_{22}$
on $U_+\cup U_-$ as
$$
  h_{11} := \inner{\nabla_{u}f_{u}}{\nu}
  ,\quad
  h_{12} := \inner{\nabla_{u}f_{v}}{\nu}= \inner{\nabla_{v}f_{u}}{\nu}
  ,\quad
  h_{22} := \inner{\nabla_{v}f_{v}}{\nu}
  ,
$$
where $\nabla$ is the Levi-Civita connection of $(M^3,\bar{g})$
and $\nabla_{u}:=\nabla_{\partial_u}$, $\nabla_{v}:=\nabla_{\partial_v}$.
Then, 
the second fundamental form $I\!I$ of $f$
is given by 
\begin{equation}\label{eq:2nd-FF}
  I\!I = h_{11}\,du^2 + 2h_{12} \, du\,dv + h_{22}\,dv^2.
\end{equation}
The extrinsic curvature function $K_{\rm ext}$ 
and the mean curvature function $H$ are written as
$$
  K_{\rm ext} := \frac{h_{11}h_{22}-(h_{12})^2}{\lambda},\qquad
  H := \frac{Eh_{22}-2Fh_{12}+Gh_{11}}{2\lambda}
$$
on $U_+\cup U_-$.
Let $K$ be the Gaussian curvature function of $ds^2$ 
defined on $\Sigma_+ \cup \Sigma_-$.
Set a smooth function $c_{\bar{g}}$ on $\Sigma_+ \cup \Sigma_-$
as 
\begin{equation}\label{eq:ambient}
  c_{\bar{g}}(p):= K_{\bar{g}}(T_p\Sigma)
  \qquad
  \left( p\in \Sigma_+ \cup \Sigma_- \right),
\end{equation}
where $K_{\bar{g}}$ is the sectional curvature 
of the Lorentzian manifold $(M^3,\bar{g})$.
Then, by the Gauss equation,
$K$ is given by 
$$
  K 
  = \begin{cases}
      -K_{\rm ext} + c_{\bar{g}} & \text{on $\Sigma_+$}, \\
      K_{\rm ext} + c_{\bar{g}} & \text{on $\Sigma_-$}. 
     \end{cases}
$$
Hence, on $U_+\cup U_-$, 
the Gaussian curvature $K$ is written as
\begin{equation}\label{eq:GaussCurvature}
  K 
  = \frac{-h_{11}h_{22}+(h_{12})^2}{|\lambda|} + c_{\bar{g}}.
\end{equation}

Let $p\in LD$ be a non-degenerate lightlike point. 
On an admissible coordinate system $(U;u,v)$
centered at $p$, there exists 
a smooth vector field 
$\psi$ of $M^3$ along $f$ 
defined on $U$
such that
$  f_u\times f_v
  = \sigma\sqrt{E} f_v+\lambda\, \psi,$
where 
$\sigma = \pm 1$.

Let $U_+$ (resp.\ $U_-$) be the set of 
spacelike points 
(resp.\ timelike points) on $U$.
Since $\nu=(f_u\times f_v)/\sqrt{|\lambda|}$ 
is a unit normal vector field on $U_+\cup U_-$,
we have
$
  \nu=
  (\sigma \sqrt{E}f_v+\lambda \psi)/\sqrt{|\lambda|}.
$
Hence, the second fundamental form $I\!I$ 
defined on on $U_+\cup U_-$ is written as
$I\!I = h_{11}\,du^2 + 2h_{12} \, du\,dv + h_{22}\,dv^2$,
where 
$$
    h_{ij} 
  :=\inner{\nabla_{u_i} f_{u_j} }{\nu}
  =\frac{1}{\sqrt{|\lambda|}}
     \left(\sigma \sqrt{E}\inner{\nabla_{u_i} f_{u_j}}{f_v} 
      + \lambda \,\phi_{ij} \right)
  \qquad(i,j=1,2),
$$
$u_1=u$, $u_2=v$, and 
$\phi_{ij}$ is the functinon given by \eqref{eq:aij}
for $i,j=1,2$.
We set smooth functions 
$\hat{h}_{ij}$ $(i,j=1,2)$ as 
$\hat{h}_{ij}:=\sqrt{|\lambda|} h_{ij}$. 
By 
$\inner{\nabla_u f_{v}}{f_v}=G_u/2$,
$\inner{\nabla_v f_{v}}{f_v}=G_v/2$
and \eqref{eq:fuufv},
it follows that 
\begin{equation}\label{eq:LMN-hat}
\begin{split}
  &\hat{h}_{11}
  =\sigma \sqrt{E}\left(F_u-\frac{E_v}{2}\right)
    +\lambda \phi_{11},\\
  \hat{h}_{12}
  =\sigma &\frac{\sqrt{E}G_u}{2}
  +\lambda \phi_{12},\quad
  \hat{h}_{22}
  =\sigma \frac{\sqrt{E}G_v}{2}+\lambda \phi_{22}.
\end{split}
\end{equation}
As in \eqref{eq:GaussCurvature},
the Gaussian curvature $K$ is given by 
$
  K 
  = (-h_{11}h_{22}+(h_{12})^2)/|\lambda| + c_{\bar{g}}
  = (-\hat{h}_{11}\hat{h}_{22}+(\hat{h}_{12})^2 + \lambda^2 c_{\bar{g}} )/\lambda^2
$
on $U_+\cup U_-$,
where $c_{\bar{g}}$ is a smooth function on $U_+\cup U_-$ 
defined as \eqref{eq:ambient}.
Hence
\begin{equation}\label{eq:Khat-DEF}
  \hat{K}:=\lambda^2 K
  =-\hat{h}_{11}\hat{h}_{22}+(\hat{h}_{12})^2 +\lambda^2 c_{\bar{g}}
\end{equation}
is a smooth function on $U$.
By \eqref{eq:LMN-hat}, $\hat{K}$ is written as 
\begin{multline}\label{eq:K-hat}
  \hat{K}
  = \frac{E}{4}\left( G_u^2 + G_v(E_v-2F_u) \right) \\
  + \frac{\sigma}{2} \lambda \sqrt{E}
    \left( -G_v \phi_{11} 
    + 2G_u \phi_{12} +\left(E_v - 2F_u\right) \phi_{22} \right)\\
  +\lambda^2 ((\phi_{12})^2-\phi_{11}\phi_{22} + c_{\bar{g}}).
\end{multline}

\begin{proposition}\label{prop:lambda-K}
Let $f:\Sigma\to M^3$ be a mixed type surface
and $p\in LD$ a lightlike point of the first kind. 
Take a specially adapted coordinate neighborhood 
$(U;u,v)$ of $p=(0,0)$.
Let $\lambda$ 
be the discriminant function 
$\lambda:=EG-F^2$ on $(U;u,v)$,
and $K$ be the Gaussian curvature on $U\setminus LD$.
Then, $\hat{K}:=\lambda^2 K$ is smoothly extended to $U$,
and expressed as
\begin{equation}\label{eq:K-hat-SP}
  \hat{K}(u,v)
  = -\frac1{2}\kappa_L(u) 
     + \frac1{2}v 
       \left( \kappa_N(u) - \kappa_B(u) 
        +\kappa_L(u)\varphi(u)
        \right)
     + v^2 \hat{K}_0(u,v),
\end{equation}
where $\hat{K}_0(u,v)$
{\rm (}resp.\ $\varphi(u)${\rm )} is 
a smooth function defined
on a neighborhood of 
$p=(0,0)$
{\rm (}resp.\ $0${\rm )}.
\end{proposition}

\begin{proof}
By \eqref{eq:Khat-DEF}, $\hat{K}$ is a smooth function on $U$.
Using \eqref{eq:K-hat},
we shall prove \eqref{eq:K-hat-SP}.
As we have seen in the proof of Lemma \ref{lem:kappa-N-sp},
the $u$-axis $\gamma(u):=(u,0)$ 
gives a characteristic curve, and 
$\eta:=\partial_v$ is a null vector field. 
Since $E(u,0)=1$ and $F(u,0)=G(u,0)=0$,
there exist smooth functions 
$\hat{E}$, $\hat{F}$, $\hat{G}$ 
on a neighborhood of $p$ such that 
$$
  E(u,v)=1 + v\, \hat{E}(u,v),\quad
  F(u,v) = v\, \hat{F}(u,v),\quad
  G(u,v) = v\, \hat{G}(u,v)
$$
holds.
Then $\lambda=EG-F^2$ is expressed as 
(cf.\ \eqref{eq:lambda-sp})
\begin{equation}\label{eq:lambda-hat}
  \lambda = v\,\hat{\lambda}\qquad
  ( \hat{\lambda}:= E \hat{G} -v \hat{F}^2).
\end{equation}

By Proposition \ref{prop:kappa_L-adap},
$  2\kappa_L(u) = - E_v(u,0)$
holds. Thus we have 
$\hat{E}(u,0) = -2 \kappa_L(u)$.
Hence, there exists a smooth function  
$\check{E}$
on a neighborhood of $p$ such that 
$
  \hat{E}(u,v) = -2 \kappa_L(u) + v\, \check{E}(u,v)
$
holds.
Namely, it holds that
\begin{equation}\label{eq:Echeck}
  E(u,v)=1 -2 v\, \kappa_L(u) + v^2 \check{E}(u,v).
\end{equation}
Similarly,
by \eqref{eq:FG-hat-v},
there exist smooth functions 
$\check{F}$, $\check{G}$ 
on a neighborhood of $p$ such that 
\begin{equation}\label{eq:FGcheck}
  F(u,v) = v\, F_v(u,0) + v^2 \check{F}(u,v),\quad
  G(u,v) = v + v^2 \check{G}(u,v)
\end{equation}
holds.
Differentiating \eqref{eq:Echeck},
$\check{E}(u,0)=E_{vv}(u,0)/2$ holds.
Hence, by \eqref{eq:kappa_B_axis}, we have
$
  \check{E}(u,0) 
  = -\kappa_B(u) + F_{uv}(u,0) 
      -\kappa_L(u) ( G_{vv}(u,0)-2F_v(u,0)^2 )/5.
$
Then, there exists a smooth function  
$\tilde{E}$
on a neighborhood of $p$ such that
\begin{multline}\label{eq:Etilde}
  \check{E}(u,v) 
  = -\kappa_B(u) + F_{uv}(u,0) \\
      -\frac1{5}\kappa_L(u) \left( G_{vv}(u,0)-2F_v(u,0)^2 \right)
    + v \, \tilde{E}(u,v)
\end{multline}
holds.
Similarly, for $\check{G}$, there exists a smooth function 
$\tilde{G}$ on a neighborhood of $p$ such that 
$\check{G}(u,v)=\check{G}(u,0)+v\tilde{G}(u,v)$. 
Thus it follows that $G(u,v)=v+v^2(\check{G}(u,0)+v\tilde{G}(u,v))$.
Moreover, by \eqref{eq:Echeck},
we have
$\sqrt{E(u,0)}=1$ and
$\partial_v(\sqrt{E(u,v)})|_{v=0} =-\kappa_L(u)$.
Hence, 
there exists a smooth function  
$\hat{\rho}$
on a neighborhood of $p$ such that
\begin{equation}\label{eq:routE}
  \sqrt{E(u,v)} 
  = 1 -v\,\kappa_L(u) +v^2 \hat{\rho}(u,v)
\end{equation}
holds.
With respect to $\phi_{11}=\inner{\nabla_uf_{u}}{\psi}$,
applying Lemma \ref{lem:kappa-N-sp},
there exists a smooth function  
$\hat{\alpha}$
on a neighborhood of $p$ such that
\begin{equation}\label{eq:Ahat}
  \sigma \phi_{11}(u,v) 
  = -\kappa_N(u) + v \, \hat{\alpha}(u,v)
\end{equation}
holds.
Substituting 
\eqref{eq:lambda-hat}, 
\eqref{eq:Echeck}, 
\eqref{eq:FGcheck}, 
\eqref{eq:Etilde}, 
\eqref{eq:routE}, and 
\eqref{eq:Ahat} 
into \eqref{eq:K-hat},
we have \eqref{eq:K-hat-SP}
with
$
  \varphi(u) 
  = 2\kappa_L(u) -G_{vv}(u,0)/5-2\check{G}(u,0) +2F_v(u,0)^2/5
     -2 \sigma \,\phi_{22}(u,0).
$
\end{proof}

Hence, by Proposition \ref{prop:lambda-K},
we can describe the behavior of the Gaussian curvature near 
a lightlike point of the first kind in terms of 
the invariant, $\kappa_L$, $\kappa_N$ and $\kappa_B$ as follows.

\begin{theorem}\label{thm:boundedness}
Let $f:\Sigma\to M^3$ be a mixed type surface
and $p\in LD$ a lightlike point of the first kind.
Then, the following holds. 
\begin{itemize}
\item[{\rm (i)}]
If $\kappa_L(p)\ne0$,
then $K$ is unbounded near $p$.
More precisely, 
if $\kappa_L(p)>0$
{\rm (}resp.\ $\kappa_L(p)<0${\rm )},
then $K$ diverges to $-\infty$
{\rm (}resp.\ $+\infty${\rm )} at $p$.
\item[{\rm (ii)}] 
If $\kappa_L=0$ holds along the characteristic curve 
on a neighborhood of $p$,
and $\kappa_N(p)\ne\kappa_B(p)$,
then $K$ is unbounded near $p$, 
and changes sign between the two sides of the characteristic curve.
\item[{\rm (iii)}]
The Gaussian curvature $K$ is bounded 
on a neighborhood $U$ of $p$ if and only if 
\begin{equation}\label{eq:K-bdd}
\kappa_L=0
\quad \text{and} \quad
\kappa_N=\kappa_B
\end{equation}
hold along the characteristic curve in $U$.
\end{itemize}
\end{theorem}

\begin{proof}
Assertions (i) and (ii)
are immediate consequences of Proposition \ref{prop:lambda-K}.
So we shall prove (iii).
By \eqref{eq:Khat-DEF}, 
$\hat{K} = v^2 \hat{\lambda}^2 K$.
Hence, by \eqref{eq:K-hat-SP},
if $K$ is bounded on a neighborhood of $p$,
then 
$$
  \kappa_L(u)=\kappa_N(u) - \kappa_B(u) 
        +\kappa_L(u)\varphi(u)=0
$$
holds, which implies \eqref{eq:K-bdd}.
Conversely, if \eqref{eq:K-bdd} holds,
\eqref{eq:K-hat-SP} yields that
$
  K(u,v) = \hat{K}_0(u,v)/\hat{\lambda}(u,v)^2
$
holds whenever $v\ne0$.
This implies the Gaussian curvature is bounded.
\end{proof}

As a corollary of Theorem \ref{thm:boundedness},
we prove that the positivity or negativity of $\kappa_L$
affects the shape of the surface here.
In the case of fronts in the Euclidean $3$-space $\R^3$,
it was proved in \cite[Corollary 1.18]{SUY1} that
a cuspidal edge with positive (resp.\ negative)
singular curvature $\kappa_s$ 
looks like positively curved (resp.\ negatively curved).
A similar result holds in the case of 
the L-singular curvature $\kappa_L$:

\begin{corollary}\label{cor:shape}
Let $f:\Sigma\to \R^3_1$ be a mixed type surface
in the Lorentz-Minkowski $3$-space $\R^3_1$,
$p\in LD$ a lightlike point of the first kind,
and $\kappa_L(p)$ the L-singular curvature at $p$. 
If $\kappa_L(p)$ is negative,
then the image of $f$ is saddle shaped near $p$.
On the other hand, 
if $\kappa_L(p)$ is positive,
then the image of $f$ is locally convex near $p$,
namely, locally lying on one side of its affine tangent plane.
\end{corollary}

\begin{proof}
Denote by $K_{\rm euc}$ the Gaussian curvature 
of $f$ as the surface in the Euclidean space $\R^3$.
It suffices to show that, 
if $\kappa_L(p)$ is positive (resp.\ negative),
then $K_{\rm euc}(p)$ is also positive (resp.\ negative).
By the assertion (i) of Theorem \ref{thm:boundedness},
if $\kappa_L(p)>0$, 
$K$ diverges to $-\infty$ at $p$.
It is known that 
$K$ and $K_{\rm euc}$ have opposite sign
(for example, see \cite{Akamine}).
Hence, $K_{\rm euc}(p)$ is positive.
In the case of $\kappa_L(p)<0$, 
a similar argument implies the desired result.
\end{proof}

\begin{example}\label{ex:N-curved}
Set $f:S^1\times (0,\infty)\to \R^3_1$ as
$$
  f(u,v)=\left(\frac{\cos u}{\cosh v},\frac{\sin u}{\cosh v},v-\tanh v\right),
$$
which is known as Beltrami's pseudosphere
as a surface of constant negative curvature 
$K_{\rm euc}=-1$ in the Euclidean $3$-space $\R^3$.
As a surface in $\R^3_1$,
$\gamma(t):=(t,\log(\sqrt{2}+1))$
is the characteristic curve 
consisting of lightlike points of the first kind.
The image $\hat{\gamma}(t)=f(\gamma(t))$
is written as 
$\hat{\gamma}(t)=(\cos t, \sin t,\sqrt{2}\log(\sqrt{2}+1)-1)/\sqrt{2}.$
By a direct calculation, we have
$$
\kappa_L(t)=-\frac{1}{\sqrt[3]{2}},\quad
\kappa_N(t)=-\sqrt[3]{2},\quad
\kappa_G(t)=0,\quad
\kappa_B(t)=-\frac{4}{5}\sqrt[3]{2}.
$$
In particular, $\kappa_L$ is negative, and hence, 
Corollary \ref{cor:shape} yields that 
this surface is saddle shaped near 
the image of the lightlike point set $LD$
(see Figure \ref{fig:PosNeg}).
\end{example}

\begin{figure}[htb]
\begin{center}
 \begin{tabular}{{c}}
  \resizebox{3cm}{!}{\includegraphics{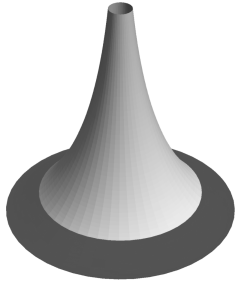}}
 \end{tabular}
 \label{fig:PosNeg}
 \caption{
 The figure of Beltrami's pseudosphere.
 As a surface in $\R^3$,
 Beltrami's pseudosphere 
 has negative curvature.
 The dark (resp.\ bright) part is the image 
 of the spacelike (resp.\ timelike) point sets 
 of the surface, and
 the boundary points of the dark and bright parts
 consist of the image of the lightlike point sets $LD$.
 As we see in Example \ref{ex:N-curved},
 every lightlike point is of the first kind
 and has the negative lightlike singular curvature 
 $\kappa_L(p)=-1/\sqrt[3]{2}$.
 This examples verifies Corollary \ref{cor:shape}.
 }
\end{center}
\end{figure}

Now, we prove Theorem \ref{thm:introB} in the introduction.

\begin{proof}[Proof of Theorem \ref{thm:introB}]
The second assertion of Theorem \ref{thm:introB}
is a direct conclusion of (iii) in Theorem \ref{thm:boundedness}.
Hence, we prove the first assertion 
of Theorem \ref{thm:introB} here.
Let $\gamma(t)$ $(|t|<\ep)$ be a
characteristic curve passing through 
$p=\gamma(0)$.
Assume that $p$ is not of the first kind.
If $p$ is a lightlike point of the admissible second kind,
then there exists a sequence $\{p_n\}_{n\in \N}$
consisting of lightlike points of the first kind
such that $\lim_{n\to \infty}p_n=p$.
The assertion (i) of Theorem \ref{thm:introA}
yields that
the L-singular curvature $\kappa_L$
cannot be bounded near $p =\gamma(0)$.
However, 
since $K$ is bounded on a neighborhood of $p$,
Theorem \ref{thm:boundedness} implies that 
$\kappa_L(p_n)=0$ for sufficiently large $n\in \N$,
which is a contradiction.
Hence, $p$ cannot be of the admissible second kind.

Thus, suppose that $p$ is an $L_{\infty}$-point.
Then, there exists $\ep >0$ such that
$\gamma(t)$ is not of the first kind for $|t|<\ep$.
Hence, $\gamma'(t)$ gives a null vector field along $\gamma(t)$.
If we take a coordinate system $(U;u,v)$ centered at 
$p=(0,0)$ such that $\eta:=\partial_v$ is a null vector field,
the image of $\gamma(t)$ coincides with the $v$-axis.
Changing the parameter $t$ if necessary, 
we may assume that $\gamma(t)=(0,t)$ holds.
Since $\eta=\partial_v$ is a null vector field,
we have $F(0,v)=G(0,v)=0$.
Therefore, there exist smooth functions 
$\hat{F}$, $\hat{G}$ 
on a neighborhood of $p=(0,0)$ such that 
$F(u,v) = u\, \hat{F}(u,v)$ and
$G(u,v) = u\, \hat{G}(u,v)$.
Then we have
$
  \lambda
  =EG-F^2
  = u(E\hat{G}-u\hat{F}^2).
$
Hence $\lambda_v(0,v)=0$ holds.
The non-degeneracy of $p$ implies 
$
  0\ne \lambda_u(0,v)=E(0,v)\hat{G}(0,v).
$
By \eqref{eq:K-hat}, we have 
$$
  \hat{K}(0,v)
  = \frac1{4}E(0,v)\hat{G}(0,v)^2>0.
$$
Therefore, the Gaussian curvature $K=\hat{K}/\lambda^2$ 
diverges to $\infty$ at $p=(0,0)$,
which is a contradiction.
Thus, $p$ must be of the first kind.
\end{proof}

In the latter part of the proof of Theorem \ref{thm:introB},
we have shown the following.

\begin{corollary}\label{cor:K-infinity}
The Gaussian curvature $K$ diverges to $\infty$ at an $L_{\infty}$-point.
\end{corollary}

As introduced in Fact \ref{fact:HKKUY},
if the mean curvature $H$ is bounded on a neighborhood of 
a non-degenerate lightlike point $p$,
then $p$ is an $L_{\infty}$-point.
Then, Corollary \ref{cor:K-infinity} yields that
the Gaussian curvature $K$ diverges to $\infty$ at $p$.
This assertion was proved in \cite[Proposition 4.8]{UY_geloma}
(cf.\ \cite{Akamine} for the case of $H=0$).

Now, we prove Corollary \ref{cor:introC} in the introduction.
In the case of fronts in the Euclidean $3$-space $\R^3$,
it was proved in \cite[Corollary B]{NUY} that
the limiting normal curvature $\kappa_{\nu}$ of 
a cuspidal edge is an extrinsic invariant.
Similarly, in \cite{Honda-deform},
it is proved that
the lightlike normal curvature $\kappa_N$
at a lightlike point $p$ of the first kind
is an extrinsic invariant.
More precisely,
let $f : (\Sigma,p)\to \R^3_1$
be a germ of a mixed type surface
in the Lorentz-Minkowski $3$-space $\R^3_1$.
Assume that $f$ is real analytic, 
$p$ is a lightlike point of the first kind,
and
$
  \kappa_L(p)\ne0.
$
Then, there exists a real analytic germ of a mixed type surface
$\bar{f} : (\Sigma,p)\to \R^3_1$ such that 
\begin{itemize}
\item
$f$ and $\bar{f}$ has the same first fundamental form
(namely, $\bar{f}$ is isometric to $f$),
and 
\item
$\kappa_N(p)\ne\bar{\kappa}_N(p)$, 
where $\kappa_N(p)$ (resp.\ $\bar{\kappa}_N(p)$) 
is the lightlike normal curvature of $f$ (resp.\ $\bar{f}$).
\end{itemize}
However, in the case of $\kappa_L=0$,
the existence of such an $\bar{f}$ cannot be proved 
in the method of \cite{Honda-deform}.
We also remark that the lightlike geodesic torsion
$\kappa_G$ is also an extrinsic invariant \cite{Honda-deform}.

\begin{proof}[Proof of Corollary \ref{cor:introC}]
Take a specially adapted coordinate system
$(U;u,v)$ as in the proof of Theorem \ref{thm:boundedness}.
If $\kappa_L(u)=0$, \eqref{eq:K-hat-SP} yields that
$$
  \hat{K}(u,v)
  = \frac1{2}v \left( \kappa_N(u) - \kappa_B(u)\right)
    + v^2 \hat{K}_0(u,v),
$$
and hence 
$
  \kappa_N(u)
  = \kappa_B(u) + 2 \hat{K}_v(u,0)
$
holds.
Since $\kappa_B$ and $\hat{K}$ are intrinsic, 
we may conclude that 
$\kappa_N$ is also an intrinsic invariant.
\end{proof}

\begin{remark}\label{rem:flat-CE}
Such a phenomenon does not occur 
in the case of wave fronts.
Namely, there exist examples of wave fronts such that
the singular curvature $\kappa_s$ is identically zero 
$\kappa_s=0$ along the singular set,
but $\kappa_{\nu}$ is still extrinsic.
For example, the real analytic germ of a wave front
$f:(\R^2,0)\to \R^3$ defined by
$f(u,v)=((1+v^3)\cos u,(1+v^3)\sin u,v^2)$
has cuspidal edge with $\kappa_s=0$ along the singular set $v=0$.
We can check that $\kappa_\nu(0)\ne0$.
And hence, by the method used in the proof of \cite[Corollary B]{NUY},
we can prove the existence of 
a real analytic germ of a wave front
$\bar{f}:(\R^2,0)\to \R^3$
such that is $\bar{f}$ is isometric to $f$ but 
$\kappa_\nu(0)\ne\bar{\kappa}_\nu(0)$.
\end{remark}

In the rest of this section,
we shall consider the behavior of umbilic points 
of mixed type surfaces
with bounded Gaussian curvature.
Let $f: \Sigma \to M^3$ be a mixed type surface.
On the timelike point set 
$\Sigma_-:=\{p\in \Sigma \,;\, (ds^2)_p \text{ is indefinite} \}$,
the restriction $f_-:=f|_{\Sigma_-}$ is a timelike surface.
Take a unit normal vector field $\nu$ of $f_-$.
The shape operator $S:=(df_-)^{-1}\circ(-d\nu)$ 
satisfies one of the following:
\begin{itemize}
\item[(i)]
$S$ is diagonalizable over $\R$.
(In this case, $H^2-K\geq 0$ holds.)
\item[(ii)]
$S$ is diagonalizable over $\C\setminus\R$.
(In this case, $H^2-K< 0$ holds.)
\item[(iii)]
$S$ is not diagonalizable over $\C$.
(In this case, $H^2-K=0$ holds.)
\end{itemize}
Then, 
a point $p\in \Sigma_-$ is umbilic if and only if 
$S$ satisfies the condition (i) and $H(p)^2-K(p)=0$ holds.
In the case (i) (resp.\ (ii)),
the principal curvatures are real valued
(resp.\ non-real complex valued).

As we saw earlier,
if $H$ is bounded on a neighborhood of 
an non-degenerate lightlike point $p$,
then $K$ diverges to $\infty$ at $p$.
Namely, $H^2-K$ diverges to $-\infty$ at $p$.
Therefore, we have:
\[
\begin{minipage}{0.9\linewidth}
{\it 
If 
$H$ 
is bounded on a neighborhood of 
an non-degenerate lightlike point $p$,
then umbilical points do not accumulate to $p$.
Moreover, there exists an open neighborhood 
$U$ of $p$ such that 
the principal curvatures are non-real complex valued
on the timelike point set $U_-:=\Sigma_-\cap U$.
}\end{minipage}
\]

In the case of bounded Gaussian curvature,
we have the following.

\begin{corollary}\label{cor:umbilic}
If the Gaussian curvature 
$K$ is bounded on a neighborhood of 
an non-degenerate lightlike point $p$,
then umbilical points do not accumulate to $p$.
Moreover, there exists an open neighborhood 
$U$ of $p$ such that 
the principal curvatures are real valued
on the timelike point set $U_-:=\Sigma_-\cap U$.
\end{corollary}

\begin{proof}
By Theorem \ref{thm:introB}, $p$ must be of the first kind.
By Fact \ref{fact:HKKUY},
the mean curvature $H$ is unbounded near $p$.
Namely, 
$H^2$ diverges to $\infty$ at $p$.
Hence, 
so does $H^2-K$.
Thus, there exists an open neighborhood 
$U$ of $p$ such that
$H^2-K>0$ holds on $U_*:=U_+\cup U_-$,
which gives the desired result.
\end{proof}

\section{Monge form for mixed type surfaces}
\label{sec:monge}
A Monge form of wave fronts in $\R^3$ at the cuspidal edge singularity
was given in \cite{MS}.
We show the following proposition which can be
regarded as a Monge form for mixed type surface
at a lightlike point of the first kind.
\begin{proposition}\label{prop:monge}
Let $f: \Sigma\to \R^3_1$ be a mixed type surface,
and let
$p\in \Sigma$ be a lightlike point of the first kind
satisfying $f(p)=0$.
Then there exist 
a local coordinate neighborhood $((U;u,v),\psi)$ of $p$,
a neighborhood $V\subset\R^3_1$ of $0$ and
a Lorentz motion $\Psi:V\to\Psi(V)$ 
such that 
{\allowdisplaybreaks
\begin{align}
\label{eq:monge}
\begin{aligned}
&\Psi\circ f\circ\psi^{-1}(u,v)\\
&\hspace{2mm}
=\Bigg(
u+uva_1(u)+uv^2a_3(u)+v^3b_1(u,v),\ 
v+u^2a_2(u),\\
&\hspace{4mm}
\int_0^u
\dfrac{u(a_1(u)+2a_2(u)+ua_2'(u))}
{\sqrt{u^2a_1(u)^2+1}}\,du
+v\sqrt{u^2a_1(u)^2+1}+v^2b_2(u,v)
\Bigg),
\end{aligned}
\end{align}}%
where 
$a_1,a_2,a_3$ are functions of one variable,
$a_2'(u)=da_2(u)/du$,
and
$b_1,b_2$ are functions of two variables satisfying
$b_2(0,0)=1/4$.
\end{proposition}
We remark that a coordinate system satisfying 
the conditions in the above proposition is adapted.
\begin{proof}
We take an adapted coordinate neighborhood $(u,v)$ near $p$.
By a Lorentz motion and a suitable change of $u\mapsto t\,u$ $(t\in \R)$, 
we may assume that
$f_u(0,0)=(1,0,0)$, 
$f_v(0,0)=(0,c,c)$ $(c>0)$.
Remark that $c$ can be chosen arbitrarily.
We set $f(u,v)=(f_1(u,v),f_2(u,v),f_3(u,v))$.
Then there exist 
functions of one variable $f_{11},f_{12}$ and 
a function of two variables $f_{13}$ such that
$f_1(u,v)=f_{11}(u)+vf_{12}(u)+v^2f_{13}(u,v)$.
We set a new coordinate system $(\tilde u,\tilde v)$ by
$\tilde u=f_{11}(u)+v^2f_{13}(u,v)$ and
$\tilde v=v$.
Then we see $\{v=0\}=\{\tilde v=0\}$ and
$\partial_v=\partial_{\tilde v}$ on $\{v=0\}$.
Thus $(\tilde u,\tilde v)$ is an adapted coordinate system.
Setting $u=g_1(\tilde u,\tilde v)$,
we see the first component of 
$f(\tilde u,\tilde v)$ is
$
\tilde u+\tilde vf_{12}(g_1(\tilde u,\tilde v)).
$
Thus we may assume that $f(u,v)$ is given by the form
$$
f(u,v)
=
(u+vf_{12}(u,v),f_2(u,v),f_3(u,v)).
$$
By the same argument together with
$(f_2)_u(0,0)=0$, there exist functions $f_{21}$ and
$f_{22}$ such that
$f_2(u,v)=u^2f_{21}(u)+cvf_{22}(u,v)$
$(f_{22}(0,0)=1)$.
We set a new coordinate system $(\tilde u,\tilde v)$ by
$\tilde u=u$ and
$\tilde v=vf_{22}(u,v)$.
Then
$(\tilde u,\tilde v)$ is an adapted coordinate system,
and the first and the second components of 
$f(\tilde u,\tilde v)$ are 
$\tilde u+\tilde v(h_1(\tilde u,\tilde v))$
and
$\tilde u^2f_{21}(\tilde u)+c\tilde v$,
respectively.
Thus we may assume that $f(u,v)$ is given by the form
$$
f(u,v)
=
(u+uva_1(u)+uv^2a_3(u)+v^3b_1(u,v),cv+u^2a_2(u),f_3(u,v)).
$$
Since $(u,v)$ is an adapted coordinate system,
$df(\partial_v)$ is lightlike on $\{v=0\}$, 
it holds that 
$
F(u,0)=G(u,0)=0.
$
Moreover, since $p$ is non-degenerate, $G_v(0,0)\ne0$.
Thus
\begin{align*}
&ua_1(u)+uc(2a_2(u)+ua_2'(u))-(f_3)_u(u,0)(f_3)_v(u,0)=0,\\
&u^2a_1(u)^2+c^2-((f_3)_u(u,0))^2=0.
\end{align*}
Since $(f_3)_v(0,0)=c>0$, we have
$
(f_3)_v(u,0)=\sqrt{u^2a_1(u)^2+c^2}.
$
On the other hand,
$$
(f_3)_u(u,0)
=
\dfrac{ua_1(u)+uc(2a_2(u)+ua_2'(u))}{(f_3)_v(u,0)}
$$
holds. Thus setting
$
f_3(u,v)=
f_{31}(u)+vf_{32}(u)+v^2f_{33}(u,v),
$
we have 
$f_{32}(u)=\sqrt{u^2a_1(u)^2+c^2}$,
and
$$
f_{31}(u)=
\int_0^u \dfrac{ua_1(u)+uc(2a_2(u)+ua_2'(u))}{\sqrt{u^2a_1(u)^2+c^2}}\,
du.$$
Thus we may assume that $f(u,v)$ is given by the form
\begin{align*}
f(u,v)
=&
\Bigg(u+uv\alpha_1(u)+uv^2\alpha_3(u)+v^3\beta_1(u,v),\ 
cv+u^2\alpha_2(u),\\
&\hspace{-5mm}
\int_0^u \dfrac{u\alpha_1(u)+uc(2\alpha_2(u)+u\alpha_2'(u))}
{\sqrt{u^2\alpha_1(u)^2+c^2}}\,du
+v\sqrt{u^2\alpha_1(u)^2+c^2}+v^2f_{33}(u,v)\Bigg),
\end{align*}
where $G_v(0,0)\ne0$, it holds that $f_{33}(0,0)\ne0$.
We set 
$k=\sqrt{|f_{33}(0,0)|}$ and
$\sigma=\sgn(f_{33}(0,0))=\pm1$.
Setting $\overline{v}=2\sigma k\,v$, we see
{\allowdisplaybreaks
\begin{align}
\label{eq:normal200}
\begin{aligned}
f(u,\overline{v})
=&
\Bigg(u + u\overline{v}\frac{\alpha_1(u)}{2\sigma k}
+ u\overline{v}^2 \frac{\alpha_3(u)}{4k^2}
+\overline{v}^3
\frac{1}{8\sigma k^3}\beta_1\left(u,\frac{\overline{v}}{2\sigma k}\right),\\
&\frac{1}{2\sigma k}c\overline{v}+u^2\alpha_2(u),\ 
\int_0^u \dfrac{u\alpha_1(u)+uc(2\alpha_2(u)+u\alpha_2'(u))}
{\sqrt{u^2\alpha_1(u)^2+c^2}}\,du\\
&\hspace{15mm}+\frac{1}{2\sigma k}\overline{v}\sqrt{u^2\alpha_1(u)^2+c^2}
+\overline{v}^2
\frac{1}{4k^2}f_{33}\left(u,\frac{\overline{v}}{2\sigma k}\right)
\Bigg).
\end{aligned}
\end{align}}%
Finally, setting 
\begin{align*}
&c=2k,\quad
a_1(u)=\dfrac{\alpha_1(u)}{2\sigma k},\quad
a_2(u)=\sigma \alpha_2(u),\quad
a_3(u)=\dfrac{\alpha_3(u)}{4k^2},\\
&\hspace{10mm}
b_1(u,\overline{v})
=\frac{1}{8\sigma k^3}
   \beta_1\left(u,\frac{\overline{v}}{2\sigma k}\right),\quad
b_2(u,\overline{v})
= \frac{1}{4\sigma k^2}
   f_{33}\left(u,\frac{\overline{v}}{2\sigma k}\right),
\end{align*}
the right-hand side of \eqref{eq:normal200} is
\begin{align*}
&\Bigg(u+u\overline{v}a_1(u)
+u\overline{v}^2a_3(u)
+\overline{v}^3b_1(u,\overline{v}),\ 
\sigma(\overline{v}+u^2a_2(u)),\\
&\hspace{2mm}
\sigma\left(\int_0^u \dfrac{u(a_1(u)+2a_2(u)+ua_2'(u))}
{\sqrt{u^2a_1(u)^2+1}}\,du
+\overline{v}\sqrt{u^2a_1(u)^2+1}
+\overline{v}^2
b_2(u,\overline{v}) \right)\Bigg),
\end{align*}
and considering a $\pi$-rotation around the
axis of the first coordinate if necessary,
$f$ can be written in the desired form.
\end{proof}

Let $f: \R^2 \to \R^3_1$ be the mixed type surface
given by the right hand side of \eqref{eq:monge}.
Since the coordinate system $(u,v)$
is adapted, applying
Proposition \ref{prop:curvature-G} and
Proposition \ref{prop:kappa_B_adapted},
we have
\begin{gather*}
  \kappa_L = a_1(0),\quad
  \kappa_N = -\frac{a_1(0)}{2}-2a_2(0),\quad
  \kappa_G = \frac{4}{3} (b_2)_u(0,0),\\
  \kappa_B = -a_2(0)+\frac{a_1(0)}{5}  \left(-5a_1(0)+ 12 (b_2)_v(0,0)-2\right)
\end{gather*}
at the origin.
On the other hand, $f$ can be expanded in the following form
$$
  f(u,v)=
  \left(u+a_1(0)u v,
  v+a_2(0)u^2,
  v+\left(\frac{1}{2}a_1(0)+a_2(0) \right)u^2+\frac1{4}v^2\right)
  +h(u,v),
$$
where $h(u,v)$ consists of the terms whose degrees 
are higher than $2$.
Thus, we have the following:

\begin{proposition}
Let $f,g: (\R^2,(0,0)) \to \R^3_1$ 
be germs of mixed type surfaces.
If their $L$-singular curvatures and 
$L$-normal curvatures coincide at $(0,0)$,
then there exist coordinate system $(u,v)$
and an isometry $A$ of $\R^3_1$ such that 
$$
  j_0^2f(u,v) = j_0^2(A\circ g)(u,v),
$$
where $j_0^2$ stands for the 
$2$-jet of the maps with respect to $(u,v)$ at $(0,0)$.
\end{proposition}

\section{The Gauss-Bonnet type formula}
\label{sec:GB}

In this section, 
applying the result by Pelletier \cite{Pelletier} and Steller \cite{Steller},
we obtain the Gauss-Bonnet type formula for mixed type surfaces
with bounded Gaussian curvature (Corollary \ref{cor:introD}).

\subsection{The Gauss-Bonnet type formula for metrics 
with varying signatures}

Let $\Sigma$ be a smooth $2$-manifold.
We call a smooth symmetric $(0,2)$-tensor
$g=\inner{~}{~}$ a {\it metric} on $\Sigma$.
A point $p\in \Sigma$ is called a {\it singular point} of $g$
if ${\rm rank}\, g_p <2$.
Set $S(g) := \{ p\in \Sigma \,;\, {\rm rank}\, g_p <2 \}$
the set of singular points.
Denote by 
$$
  \mathcal{N}_p:=\left\{
    \vect{v}\in T_p\Sigma\,;\,
    \inner{\vect{v}}{\vect{x}}=0  
    \text{~holds for any~} \vect{x}\in T_p\Sigma
  \right\}
$$
the {\it null space} of $g$ at $p$.
A singular point $p \in S(g)$ is called {\it corank one\/} 
if ${\rm dim}\,\mathcal{N}_p=1$.
A non-zero smooth vector field $\eta$
such that $\eta_p \in \mathcal{N}_p$ holds
for each $p \in S(g)$
is called a {\it null vector field}.
A metric $g$ is called {\it generic} if
the following conditions hold (cf.\ \cite[Definition 1]{Steller}):
\begin{itemize}
\item[($G_1$)]
The singular set $S(g)$ is a union of 
distinct simply closed regular smooth curves
$S_1,\dots,S_m$ $(m\geq1)$. 
Each $S_j$ $(j=1,\dots,m)$ is called a {\it singular curve}.
\item[($G_2$)]
Every singular point is corank one 
(namely, ${\rm dim}\,\mathcal{N}_p=1$ holds for $p \in S(g)$). 
\item[($G_3$)]
For any null vector field $\eta$, 
it holds that $\eta_p\inner{\eta}{\eta} \ne0$
for each singular point $p \in S(g)$.
\end{itemize}

For any smooth vector fields $X,Y,Z$ on $\Sigma$,
we set 
$\square_{X} Y(Z) : \mathfrak{X}(\Sigma)\times 
\mathfrak{X}(\Sigma)\times \mathfrak{X}(\Sigma) \to C^\infty(\Sigma)$ as
\begin{multline*}
  \square_{X} Y(Z)
  := \frac1{2} \Bigl( X\inner{Y}{Z} + Y\inner{Z}{X} - Z\inner{X}{Y}\\
        - \inner{X}{[Y,Z]} + \inner{Y}{[Z,X]} + \inner{Z}{[X,Y]}\Bigr),
\end{multline*}
which is called the {\it Levi-Civita dual connection} 
\cite{Kossowski1987proc}.
Letting $D$ be the Levi-Civita connection
on the non-singular set $\Sigma\setminus S(g)$,
it holds that $\square_{X} Y(Z)=\inner{D_{X} Y}{Z}$.
Then, a regular curve $\gamma$ in $(\Sigma,g)$ 
is called a {\it pseudo-geodesic} if
$$
  \square_{X} X(Z) =0
$$
holds, 
where $X$ and $Z$ are non-zero vector fields along $\gamma$
such that 
$X$ is tangent to $\gamma$
and $Z$ is orthogonal to $\gamma$.
We remark that this definition does not depend on the choice of such 
vector fields $X$ and $Z$, and 
is valid even on the singular set.

Let $dA$ be the canonical volume element of $(\Sigma,g)$.
Namely, on a coordinate neighborhood $(U;u,v)$,
$dA$ is written as
$$
  dA = \sqrt{|\lambda|}\,du\wedge dv
  \qquad
  (\lambda:=EG-F^2),
$$
where $g=Edu^2+2Fdu\,dv+Gdv^2$.
We set a signature function $\sigma_g$ on $(\Sigma,g)$ as
$$
  \sigma_g(p)
  = \begin{cases}
      -1 & \text{(if $g_p$ is negative definite)}, \\
      1 & \text{(otherwise)}. 
     \end{cases}
$$
Denoting by $K$ the Gaussian curvature function 
on the non-singular set $\Sigma\setminus S(g)$,
we call
\begin{equation}\label{eq:GCWS}
  \bar{K} := \sigma_g \, K
\end{equation}
the {\it Gaussian curvature-with-sign}.
According to Pelletier \cite{Pelletier} and Steller \cite{Steller},
the following Gauss-Bonnet type formula holds:

\begin{fact}\label{fact:GB}
Let $\Sigma$ be a connected compact oriented smooth $2$-manifold
without boundary.
If $g$ is a generic metric on $\Sigma$
such that the singular set $S(g)$ is non-empty
and each singular curve in $S(g)$ is a pseudo-geodesic,
then
$$
  \int_{\Sigma} \bar{K}\,dA = 2\pi \,\chi (\Sigma)
$$
holds, where $\chi (\Sigma)$ is the Euler characteristic of $\Sigma$,
and $\bar{K}$ is the Gaussian curvature-with-sign.
\end{fact}

We remark that Steller \cite{Steller} proves 
the Gauss-Bonnet type formula
in more general situations
(admitting the intersection points of singular curves), 
see \cite{Steller}.

\subsection{The Gauss-Bonnet type formula for metrics 
for mixed type surfaces}

Now, we apply 
the Gauss-Bonnet type formula for generic metrics 
(Fact \ref{fact:GB})
to mixed type surfaces.

Let $f : \Sigma \to (M^3,\bar{g})$ be a mixed type surface.
Then, the first fundamental form 
$ds^2$ of $f$ induces a metric on $\Sigma$.
Then, a point $p\in \Sigma$ is a lightlike point of the surface $f$
if and only if $p$ is a singular point of the metric $ds^2$.
Namely, the lightlike point set $LD$ of $f$ coincides with 
the singular set $S(ds^2)$.
Hence, if $ds^2$ is generic as a metric 
satisfying the condition of the pseudo-geodesics,
we may apply the Gauss-Bonnet type formula 
(Fact \ref{fact:GB}).

Here, we translate the conditions for generic metrics 
($G_1$)--($G_3$) into our language.
As we remarked in Section \ref{sec:prelim},
every lightlike point $p\in LD$ of a mixed type surface $f$
is a corank one singular point of the metric $ds^2$
(cf.\ \eqref{lem:corank1}).
Namely, the condition ($G_2$) is always satisfied.
Then, by Proposition \ref{prop:type1},
the condition ($G_3$) holds if and only if
every lightlike point $p\in LD$ is of the first kind.
In particular, $p$ is non-degenerate,
and hence, 
each connected component of $LD$
is a regular curve,
so the condition ($G_1$) is satisfied.
In conclusion, we have the following:

\begin{lemma}\label{lem:generic-Steller}
The first fundamental form 
$ds^2$ of a mixed type surface $f : \Sigma \to M^3$
is a generic metric on $\Sigma$ if and only if
every lightlike point $p\in LD$ is of the first kind.
\end{lemma}

From now on, 
assume that every lightlike point $p\in LD$
is of the first kind.
Take a characteristic curve $\gamma(t)$ 
passing through $p=\gamma(0)$.
We shall review 
the condition that $\gamma$ 
to be a pseudo-geodesic given in \cite{Steller}.
Take a coordinate neighborhood $(U;u,v)$
of $p$ such that
$F=0$ on $U$
and 
$G(u,0)=0$ hold.
In other words,
$(u,v)$ is an orthogonal coordinate system,
such that the $u$-axis is a singular curve 
and $\partial_v$ is a null vector field.
Hence we may call such $(u,v)$
an {\it orthogonally adapted coordinate system}.
Then, the following was proved in \cite[Proposition 1]{Steller}:
\[
\begin{minipage}{0.9\linewidth}
{\it 
If $(U;u,v)$ is an orthogonally adapted coordinate system,
$\gamma(u)=(u,0)$
is a pseudo-geodesic if and only if 
$E_v(u,0)=0$.
}\end{minipage}
\]
Thus, we have the following.

\begin{lemma}
\label{lem:Prop1-Steller}
Let $f : \Sigma \to (M^3,\bar{g})$ be a mixed type surface
such that every lightlike point is of the first kind.
Then, each connected component of $LD$
is a pseudo-geodesic if and only if 
every lightlike point has vanishing lightlike singular curvature,
namely $\kappa_L(p)=0$ for any $p\in LD$.
\end{lemma}

\begin{proof}
Let $(U;u,v)$ be an orthogonally adapted coordinate system
centered at a lightlike point $p=(0,0)$ of the first kind.
By \cite[Proposition 1]{Steller}, 
the characteristic curve $\gamma(u)=(u,0)$ 
is a pseudo-geodesic if and only if 
$E_v(u,0)=0$.
On the other hand, 
by Proposition \ref{prop:kappa_L-adap},
$E_v(u,0)=0$ if and only if 
$\kappa_L(u)=0$, which gives the desired result.
\end{proof}

By Fact \ref{fact:GB} and  
Lemmas \ref{lem:generic-Steller} and \ref{lem:Prop1-Steller},
we have the following Gauss-Bonnet type formula 
for mixed type surfaces as follows:

\begin{theorem}[The Gauss-Bonnet type formula 
for mixed type surfaces]
\label{thm:GB-kappaL}
Let $f : \Sigma \to (M^3,\bar{g})$ be a mixed type surface 
in a Lorentzian $3$-manifold $(M^3,\bar{g})$,
where $\Sigma$ is 
a connected compact oriented smooth $2$-manifold
without boundary.
If every lightlike point of $f$ is of the first kind
having zero lightlike singular curvature $\kappa_L=0$,
then
$$
  \int_{\Sigma} K\,dA = 2\pi \,\chi (\Sigma)
$$
holds, where $\chi (\Sigma)$ is the Euler characteristic of $\Sigma$.
\end{theorem}

\begin{proof}
By Lemma \ref{lem:generic-Steller},
the first fundamental form $ds^2$ of the surface $f$
is a generic metric on $\Sigma$.
Also, by Lemma \ref{lem:Prop1-Steller},
each singular curve of $ds^2$ is a pseudo-geodesic.
Moreover, since $(ds^2)_p$ cannot be negative definite
at each point $p\in \Sigma$, 
the Gaussian curvature-with-sign $\bar{K}$
given in \eqref{eq:GCWS} coincides with 
the Gaussian curvature $K$ on the 
non-lightlike point set $\Sigma\setminus LD$.
Therefore, by a direct conclusion of Fact \ref{fact:GB},
we have the desired result.
\end{proof}

\begin{proof}[Proof of Corollary \ref{cor:introD}]
By Theorem \ref{thm:introB}, if $K$ is bounded, 
then every non-degenerate lightlike point must be of the first kind.
Moreover, by Theorem \ref{thm:introB},
if $K$ is bounded, $\kappa_L(p)=0$ holds 
for any lightlike point $p\in LD$ of the first kind.
Then, Theorem \ref{thm:GB-kappaL} 
gives the desired results.
\end{proof}

\begin{example}\label{ex:flat-torus}
Let $\tilde{f} : \R \times S^1 \to \R^3_1$ be an immersion
defined by 
$$
  \tilde{f}(u,v):=(u,\cos v,\sin v)
  \qquad
  (u\in \R,~v\in S^1:=\R/2\pi\Z),
$$
which is a cylinder over a circle in the timelike $yz$-plane
(Figure \ref{fig:fronts}). 
We can check that every lightlike point
of $\tilde{f}$ is non-degenerate, and $\tilde{f}$ is flat 
(i.e., $K=0$ on the non-lightlike point set).
In particular, $K$ is bounded.
By Theorem \ref{thm:introB},
we have every lightlike point
of $\tilde{f}$ is of the first kind and 
$\kappa_L=0$ along the lightlike point set $LD$.
The parallel translation 
$$
{\rm pr}(x,y,z):=(x+2\pi,y,z)
\qquad
((x,y,z)\in \R^3_1)
$$
generates a subgroup 
$\Gamma=\langle{\rm pr}\rangle$ 
of the isometry group of $\R^3_1$.
Let $M^3$ be a flat Lorentz $3$-manifold
given as a quotient 
$M^3 := \R^3_1/\Gamma$.
Then, $\tilde{f}$ induces a mixed type surface
$f : S^1\times S^1\to M^3$.
Since $\tilde{f}$ is flat, so is $f$.
And also, every singular 
lightlike point of $f$ is of the first kind and 
$\kappa_L=0$ along the lightlike point set $LD$.
Hence, this example verifies Theorem \ref{thm:GB-kappaL}
and  Corollary \ref{cor:introD}. 
\end{example}

\begin{acknowledgement}
The authors thank the referee
for careful reading and valuable comments.
They also thank 
Masaaki Umehara and Kotaro Yamada 
for their advice.
\end{acknowledgement}


\medskip

\begin{thebibliography}{20}

\bibitem{AFL}
E.\ Aguirre, V.\ Fernandez and J.\ Lafuente,
{\it On the conformal geometry of transverse 
Riemann--Lorentz manifolds},
J. Geometry and Physics {\bf 57} (2007), 1541--1547.

\bibitem{Akamine}
S. Akamine,
{\it Behavior of the Gaussian curvature 
of timelike minimal surfaces with singularities},
to appear in Hokkaido Mathematical Journal, arXiv:1701.00238.

\bibitem{AB}
B.L.\ Al'tshuler and A.O.\ Barvinski,
{\it Quantum cosmology and physics of transitions 
with a change of spacetime signature}.
Uspekhi Fiz. Nauk 166:5 (1996), 459.492; 
English transl. in
Physics-Uspekhi 39, 429.

\bibitem{FKKRSTUYY}
S.~Fujimori, Y.W.~Kim, S.-E.~Koh, W.~Rossman, 
H.~Shin, H.~Takahashi, M.~Umehara, K.~Yamada and S.-D.~Yang,  
{\it Zero mean curvature surfaces in $L^3$ containing a light-like line}, 
C.\ R.\ Math.\ Acad.\ Sci.\ Paris {\bf 350} (2012), no. 21-22, 975--978.

\bibitem{FRUYY}
S.~Fujimori, W.~Rossman, M.~Umehara, K.~Yamada and S.-D.~Yang,  
{\it Embedded triply periodic zero mean curvature surfaces of mixed type in
Lorentz-Minkowski 3-space},
Michigan Math.\ J.\ {\bf 63} (2014), no. 1, 189--207.

\bibitem{FKKRSUYY_okayama}
S.~Fujimori, Y.W.~Kim, S.-E.~Koh, W.~Rossman, 
H.~Shin, M.~Umehara, K.~Yamada and S.-D.~Yang,  
{\it Zero mean curvature surfaces in Lorentz-Minkowski 3-space and
2-dimensional fluid mechanics}, 
Math.\ J.\ Okayama Univ.\ {\bf 57} (2015), 173--200.

\bibitem{FKKRSUYY_osaka}
S.~Fujimori, Y.W.~Kim, S.-E.~Koh, W.~Rossman, 
H.~Shin, M.~Umehara, K.~Yamada and S.-D.~Yang,  
{\it Zero mean curvature surfaces in Lorentz-Minkowski 3-space 
which change type across a light-like line}, 
Osaka J.\ Math.\ {\bf 52} (2015), no. 1, 285--297.

\bibitem{FKKRUY_qj}
S.~Fujimori, Y.~Kawakami, M.~Kokubu, W.~Rossman, 
M.~Umehara and K.~Yamada, 
{\it Entire zero-mean curvature graphs of mixed type in
Lorentz-Minkowski 3-space},
Q.\ J.\ Math.\ {\bf 67} (2016), no. 4, 801--837.

\bibitem{FKKRUY_osaka}
S.~Fujimori, Y.~Kawakami, M.~Kokubu, W.~Rossman, 
M.~Umehara and K.~Yamada, 
{\it Analytic extension of Jorge-Meeks type maximal surfaces in
Lorentz-Minkowski 3-space},
Osaka J.\ Math.\ {\bf 54} (2017), no. 2, 249--272.

\bibitem{FHKUY}
S.~Fujimori, U.~Hertrich-Jeromin, M.~Kokubu, M.~Umehara and K.~Yamada, 
{\it Quadrics and Scherk towers}, 
Monatsh.\ Math.\ {\bf 186} (2018), no. 2, 249--279.

\bibitem{GKT}
D.\ Genin, B.\ Khesin and S.\ Tabachnikov,
{\it Geodesics on an ellipsoid in Minkowski space}.
Enseign.\ Math.\ {\bf 53} (2007), 307--331.

\bibitem{GR}
R.\ Ghezzi and A.O.\ Remizov,
{\it On a class of vector fields with discontinuities 
of divide-by-zero type and its applications to
geodesics in singular metrics},
J.\ Dyn.\ Control Syst.\ {\bf 18} (2012), 135--158.

\bibitem{Gu}
C.H.~Gu, 
{\it The extremal surfaces in the 3-dimensional Minkowski space}, 
Acta Math.\ Sinica (N.S.) {\bf 1} (1985), 173--180.

\bibitem{HHNSUY}
M. Hasegawa, A. Honda, K. Naokawa, K. Saji, M. Umehara and K. Yamada,
{\it Intrinsic properties of surfaces with singularities},
Internat. J. Math. {\bf 26} (2015), no. 4, 1540008, 34 pp.

\bibitem{Honda-deform}
A. Honda,
{\it Isometric deformations of mixed type surfaces in Lorentz-Minkowski space},
preprint, arXiv:1908.01967.

\bibitem{HKKUY}
  A.~Honda, M.~Koiso, M.~Kokubu, M.~Umehara and K.~Yamada,
  {\itshape Mixed type surfaces with bounded mean curvature 
  in $3$-dimensional space-times},
  Differential Geometry and its Applications {\bf 52} (2017) 64--77.

\bibitem{IzumiyaTari2010}
S.~Izumiya and F.~Tari,
{\it Self-adjoint operators on surfaces with a singular metric},
J. Dyn. Control Syst. {\bf 16} (2010), 329--353.

\bibitem{IzumiyaTari2013}
S.~Izumiya and F.~Tari,
{\it Apparent contours in Minkowski 3-space 
and first order ordinary differential equations},
Nonlinearity {\bf 26} (2013), 911--932.

\bibitem{KT}
B.~Khesin and S.~Tabachnikov,
{\it Pseudo-Riemannian geodesics and billiards},
Adv.\ Math.\ {\bf 221} (2009), 1364--1396

\bibitem{Klyachin}
V.A.~Klyachin, 
{\it Zero mean curvature surfaces of mixed type in Minkowski space}, 
Izv.\ Math.\ {\bf 67} (2003), 209--224.


\bibitem{KRSUY}
 M. Kokubu,  W. Rossman, K. Saji, M. Umehara, and K. Yamada,
 {\it Singularities of flat fronts in hyperbolic $3$-space},
 Pacific J. Math. {\bf 221} (2005), 303--351.

\bibitem{Kossowski1987proc}
  M.~Kossowski,
  {\it Pseudo-Riemannian metrics singularities and the extendability 
  of parallel transport},
  Proc. Amer. Math. Soc. {\bf 99} (1987), 147--154.

\bibitem{Kos2}
M.~Kossowski, 
{\it The Boy-Gauss-Bonnet theorems for $C^\infty$-singular surfaces 
with limiting tangent bundle}, 
Ann.\ Global Anal.\ Geom.\ {\bf 21} (2002), 19--29.

\bibitem{KosKri}
M.~Kossowski and M.~Kriele,
{\it Smooth and discontinuous signature type change in general relativity},
Class.\ Quantum Grav.\ {\bf 10} (1993), 2363--2371.

\bibitem{KosKri1}
M.~Kossowski and M.~Kriele,
{\it Transverse, type changing, pseudo-Riemannian metrics 
and the extendability of geodesics},
Proc.\ Roy.\ Soc.\ Lond.\ Ser.\ A Math.\ Phys.\ 444:1921 (1994), 297--306.

\bibitem{KosKri2}
M.~Kossowski and M.~Kriele,
{\it The Einstein equation for signature type changing spacetimes},
Proc.\ Roy.\ Soc.\ Lond.\ Ser.\ A Math.\ Phys.\ 446:1926 (1994), 115--126.

\bibitem{MS}
 L. F. Martins and K. Saji,
 {\it Geometric invariants of cuspidal edges},
 Canad. J. Math. {\bf 68} (2016), no. 2, 445--462.

\bibitem{MSUY}
L. F. Martins, K. Saji, M. Umehara and K. Yamada,
{\itshape Behavior of Gaussian curvature and
mean curvature near non-degenerate singular
points on wave fronts},
Geometry and Topology of Manifold,
Springer Proc.\ in Math.\ \& Stat. {\bf 154},
2016, Springer, 247--282.

\bibitem{Mier}
T.~Miernowski,
{\it Formes normales d'une m\'etrique mixte analytique r\'eelle g\'en\'erique},
Ann.\ Fac.\ Sci.\ Toulouse Math.\ {\bf 16} (2007), 923--946.

\bibitem{NUY}
K.~Naokawa, M.~Umehara and K.~Yamada,
{\it Isometric deformations of cuspidal edges},
  Tohoku Math.\ J.\ (2) {\bf 68} (2016), 73--90.

\bibitem{Pelletier}
 F.~Pelletier, 
 {\it Pseudo m\'etriques g\'en\'eriques et th\'eor\`eme 
 de {G}auss-{B}onnet en dimension {$2$}},
 Singularities and dynamical systems ({I}r\'aklion, 1983),
 219--238,
 North-Holland Math. Stud., 103,
 North-Holland, Amsterdam, 1985.

\bibitem{Rem-Pseudo}
A.O.~Remizov,
{\it Geodesics on 2-surfaces with pseudo-Riemannian metric: 
singularities of changes of signature},
Mat.\ Sb.\ 200:3 (2009), 75--94.


\bibitem{Rem15}
A.O.~Remizov,
{\it On the local and global properties of geodesics 
in pseudo-Riemannian metrics},
Differential Geom.\ Appl.\ {\bf 39} (2015), 36--58.

\bibitem{Rem-Tari}
A.O.~Remizov and F.~Tari,
{\it Singularities of the geodesic flow on surfaces 
with pseudo-Riemannian metrics},
Geom.\ Dedicata {\bf 185} (2016), 131--153.

\bibitem{SUY1}
  K.~Saji, M.~Umehara and K.~Yamada,
  {\it The geometry of fronts},
  Ann.\ of Math.\ (2) {\bf 169} (2009), 491--529.

\bibitem{SUY2} 
  K.~Saji, M.~Umehara and K.~Yamada,
  {\it Behavior of corank one singular points on wave fronts},
  Kyushu J.\ Math.\ {\bf 62} (2008), 259--280.

\bibitem{SUY4}
  K.~Saji, M.~Umehara and K.~Yamada,
  {\it Coherent tangent bundles and {G}auss-{B}onnet formulas for wave fronts},
  J.\ Geom.\ Anal.\ {\bf 22} (2012), 383--409.

\bibitem{SUY5}
  K.~Saji, M.~Umehara and K.~Yamada,
\emph{An index formula for a bundle homomorphism of the tangent bundle 
into a vector bundle of the same rank, and its applications},
J.\ Math.\ Soc.\ Japan {\bf 69} (2017), 417--457.

\bibitem{Sakhar}
A.D.~Sakharov,
{\it Cosmological transitions with changes in the signature of the metric},
Zh.\ Eksper.\ Teor.\ Fiz.\ 87:2 (8) (1984), 375--383.
English transl. in
Soviet Phys.\ JETP {\bf 60} (1984), 214--218.

\bibitem{Steller}
 M.~Steller, 
 {\it A Gauss-Bonnet formula for metrics with varying signature}, 
 Z.\ Anal.\ Anwend.\ {\bf 25} (2006), no. 2, 143--162.

\bibitem{Tari2012}
F.~Tari,
{\it Caustics of surfaces in the Minkowski 3-space},
Q. J. Math. {\bf 63} (2012), 189--209.

\bibitem{Tari2013}
F.~Tari,
{\it Umbilics of surfaces in the Minkowski 3-space},
J. Math. Soc. Japan {\bf 65} (2013), 723--731.

\bibitem{UY_hokudai}
 M. Umehara and K. Yamada,
 {\it Maximal surfaces with singularities in Minkowski space},
 Hokkaido Math.\ J.\ {\bf 35} (2006), 13--40.

\bibitem{UY_geloma}
 M. Umehara and K. Yamada,
 {\it Surfaces with light-like points in Lorentz-Minkowski 3-space with applications},
In: Ca\~{n}adas-Pinedo M., Flores J., Palomo F. (eds) 
Lorentzian Geometry and Related Topics. GELOMA 2016. 
Springer Proceedings in Mathematics \& Statistics,  pp 253--273 (2017),
vol 211. Springer, Cham. 

\bibitem{UY_2018}
 M. Umehara and K. Yamada,
 {\it Hypersurfaces with light-like points in a Lorentzian manifold},
 J.\ Geom.\ Anal.\ {\bf 29} (2019), 3405--3437.
\end{thebibliography}
\end{document}